\documentclass[10pt]{amsart}
\usepackage[utf8]{inputenc}
\usepackage{amsfonts}
\usepackage{graphics}
\usepackage[all, cmtip]{xy}
\usepackage{amsthm,amsfonts,amssymb,amsmath,amsxtra}
\usepackage[all]{xy}
\usepackage{amsmath}
\usepackage{flexisym}
\usepackage{mathrsfs}
\usepackage{amsfonts}
\usepackage{enumitem}
\usepackage{amsmath}
\usepackage[colorlinks]{hyperref}
\hypersetup{
  citecolor   = blue 
 }

\newcommand{\xdownarrow}[1]{%
  {\left\downarrow\vbox to #1{}\right.\kern-\nulldelimiterspace}
}

\newcommand{\rmM}{{\rm M}}

\newcommand{\Res}{{\rm Res}}
\newcommand{\Gal}{{\rm Gal}}

\newcommand{\ZZ}{\mathbb{Z}}
\newcommand{\PP}{\mathbb{P}}
\newcommand{\QQ}{\mathbb{Q}}

\newcommand{\calA}{\mathcal{A}}

\newcommand{\calF}{\mathcal{F}}
\newcommand{\calL}{\mathcal{L}}

\newcommand{\calU}{\mathcal{U}}

\newcommand{\scrG}{\mathscr{G}}

\usepackage{multicol}
\numberwithin{equation}{subsection}

\newtheorem{Theorem}{Theorem}[section]
\newtheorem{Remark}[Theorem]{Remark}
\newtheorem{Remarks}[Theorem]{Remarks}
\newtheorem{Lemma}[Theorem]{Lemma}
\newtheorem{Proposition}[Theorem]{Proposition}
\newtheorem{Corollary}[Theorem]{Corollary}

\newtheorem{Main Theorem}[Theorem]{Main Theorem}
\newtheorem{Definition}[Theorem]{Definition}

\usepackage[colorlinks]{hyperref}
\makeatletter
\renewcommand*\env@matrix[1][*\c@MaxMatrixCols c]{%
  \hskip -\arraycolsep
  \let\@ifnextchar\new@ifnextchar
  \array{#1}}
\makeatother

\newif\ifgrading

\gradingfalse

\usepackage[colorlinks]{hyperref}
\hypersetup{linkcolor=black}
\usepackage{xcolor}
\usepackage{ wasysym }

\usepackage{tikz-cd}
\newcommand{\U}{{\mathcal U}}
\newcommand{\calG}{{\mathcal G}}
\newcommand{\calO}{{\mathcal O}}
\newcommand{\Mloc}{{\rm M}^{\rm loc}}
\newcommand{\Spec}{{\rm Spec \, } }
\newcommand{\wti}{\widetilde}

\usepackage{xcolor}

\newcommand{\Addresses}{{
		\bigskip
		\footnotesize
		
		\textsc{Department of Mathematics, Universität Münster, Münster, 48149, Germany}\par\nopagebreak
		\textit{E-mail address:} \texttt{io.zachos@uni-muenster.de}\\
		
		\textsc{Department of Applied Mathematics,
University of Science and Technology Beijing, Beijing, 100083, China}\par\nopagebreak
		\textit{E-mail address:} \texttt{zhihaozhao@ustb.edu.cn}
}}	

\begin{document}

\title[Semi-Stable models for ramified Unitary Shimura Varieties]{Semi-stable and splitting models for unitary Shimura varieties over ramified
places. II.} 
	\date{}
	\author{I. Zachos and Z. Zhao}
	\maketitle

	\begin{abstract}
We consider Shimura varieties associated to a unitary group of signature $(n-1,1)$. For these varieties, we construct $p$-adic integral models over odd primes $p$ which ramify in the imaginary quadratic field with level subgroup at $p$ given by the stabilizer of a vertex lattice in the hermitian space. Our models are given by a variation of the construction of the splitting models of Pappas-Rapoport 
and they have a simple moduli theoretic description. By an explicit calculation, we show that these splitting models are normal, flat, Cohen-Macaulay and with reduced special fiber. In fact, they have relatively simple singularities: we show that a single blow-up along a smooth codimension one subvariety of the special fiber produces a semi-stable model. This also implies the existence of semi-stable models of the corresponding Shimura varieties.
	\end{abstract}
	
	\tableofcontents

\section{Introduction}\label{Intro}

\subsection{} 
The aim of this paper is to construct “good” integral models of Shimura varieties over places of bad reduction. Here, we consider Shimura varieties
associated to unitary groups of signature $(n-1,1)$ over an imaginary quadratic field $K$. These Shimura varieties are of PEL type and so they can be described as moduli spaces of abelian varieties with polarization, endomorphisms and level structure. It is desirable to define such integral models by a suitable extension of the moduli problem. Such models should be useful in various arithmetic applications.

Shimura varieties have canonical models over the reflex number field $E$. In the cases we consider here the reflex field is $E=K$. They are also expected to give rise to reasonable integral models. However, the behavior of these depends very much on the “level subgroup”. The level subgroup we consider here is determined by the choice of a vertex lattice in a hermitian space of dimension $n$. This stabilizer, by what follows below, is not connected when $n$ is even, so not parahoric. However, by using the work of Rapoport-Zink \cite{RZbook} 
we first construct $p$-adic integral models, which have simple and explicit moduli descriptions but are not always flat. The \'etale local structure of all these models is controlled by the local structure of certain simpler schemes the \textit{naive local models} which are defined in terms of linear algebra data inside the product of Grassmannian varieties. 
Inspired by the work of Pappas-Rapoport \cite{PR2}, we consider a variation of the above moduli problem (parametrizing abelian schemes), \textit{the splitting model}, 
where we add in the moduli problem an extra linear data of a flag of the Hodge filtration with some restricting properties. 
This is essentially an instance of the notion of a ``linear modification" introduced in \cite{P}. Then, we show that the splitting models are flat, normal, Cohen-Macaulay and with reduced special fiber. We also resolve their singularities which leads to regular models for these Shimura varieties 
over the $p$-adic integers $\mathbb{Z}_p$. We anticipate that our constructions will have applications to the study of arithmetic intersections of special cycles and Kudla’s program. (See for example, \cite{BHKR} and \cite{HLSY}, for some works in this direction.)

\subsection{} Let us give some details. To explain our results, we need to introduce some notation. 
Let $W$ be a $n$-dimensional $K$-vector space, equipped with a non-degenerate hermitian form $\phi$. Consider the group of unitary similitudes $G = GU_n$ for $(W,\phi)$ of dimension $n> 3$ over $K$. Let us mention here that the cases $n =2, 3$ have already been studied; see \cite[Theorem 2.6.2]{PRS} and \cite[Remark 2.6.13]{PRS} from the survey for more details. Fix a conjugacy class of homomorphisms $h: \Res_{\mathbb{C}/\mathbb{R}}\mathbb{G}_{m,\mathbb{C}}\rightarrow GU_n$ that corresponds to a Shimura datum $(G,X) = (GU_n,X_h)$ of signature $(n-1,1)$. The pair $(G,X)$ gives rise to a Shimura variety $Sh(G,X)$ over the reflex field $E=K$. Let $p$ be an odd prime number which ramifies in $K$ and set $K_1 =K_v $ where $v$ is the unique prime ideal of $K$ above $(p)$. Denote by $\mathcal{O}$ the ring of integers of $K_1$ and let $\pi$ be a uniformizer of $\mathcal{O}$. Set $V=W\otimes_{\QQ} \QQ_p$. 
We assume that the hermitian form $\phi$ is split on $V$, i.e. there is a basis $e_1, \dots, e_n$ such that 
\[
\phi(ae_i,be_{n+1-j}) = \overline{a}b\delta_{i,j} \quad \text{for  all} \quad a,b \in K_1, 
\]
where $a \mapsto \overline{a}$ is the nontrivial element of $\text{Gal}(K_1/\QQ_p)$. We denote by 
\[
\Lambda_i = \text{span}_{\calO} \{\pi^{-1}e_1, \dots, \pi^{-1}e_i, e_{i+1}, \dots, e_n\}
\]
the standard lattices in $V$. Consider the (quasi-)parahoric stabilizer subgroup
\[
P_I:=\{g\in GU_n\mid g\Lambda_i=\Lambda_i , \quad \forall i \in I\},
\]
for the nonempty subsets $I  \subset \{0, \dots, m\}$ where $m = \lfloor n/2\rfloor$ with the property that  
\begin{equation}\label{IpropertyIntro}
\text{for } n \text{ even, if } m-1 \in I \Longrightarrow m \in I.
\end{equation} 
(See also \cite{Luo}). More precisely, as in \S \ref{ParahoricSbgrs}, when $n$ is odd $P_I$ is a parahoric subgroup. When $n$ is even, $P_I$ is not a parahoric subgroup since it contains a parahoric subgroup with index 2 and the corresponding parahoric group scheme is its connected component $ P_I^\circ$. Let $\mathcal{L}_I$, as in \S \ref{Prel.1}, be the
self-dual multichain consisting of lattices $\{\Lambda_j\}_{j\in n\mathbb{Z} \pm I}$. Let $\mathcal{G}_I = \underline{{\rm Aut}}(\mathcal{L}_I)$ be the (smooth) group scheme over $\mathbb{Z}_p$ with $P_{I} = \mathcal{G}_I(\mathbb{Z}_p)$ and $G\otimes_{\mathbb{Z}_p}\mathbb{Q}_p $ as its generic fiber.   

Choose also a sufficiently small compact open subgroup $K^p$ of the prime-to-$p$ finite adelic points $G({\mathbb A}_{f}^p)$ of $G$ and set $\mathbf{K}=K^pP_{I}$. The Shimura variety  ${\rm Sh}_{\mathbf{K}}(G, X)$ 
is of PEL type and has a canonical model over the reflex field $K$. 

Next, by using the work of Rapoport and Zink \cite[Definition 6.9]{RZbook} we define the moduli scheme $\mathcal{A}^{\rm naive}_{\mathbf{K}}$ over $\mathcal{O}$ whose generic fiber agrees with ${\rm Sh}_{\mathbf{K}}(G, X)$. In particular, $\mathcal{A}^{\rm naive}_{\mathbf{K}}$ is a moduli of quadruples $(A,\iota, \bar{\lambda}, \bar{\eta})$ where $A$ is an abelian variety and $(\iota, \bar{\lambda}, \bar{\eta})$ are some additional data of polarization and level structure (see \S \ref{ShimuraVarSec} for more details). The functor $\mathcal{A}^{\rm naive}_{\mathbf{K}}$ is representable by a quasi-projective scheme over $\mathcal{O}$. The moduli scheme $\mathcal{A}^{\rm naive}_{\mathbf{K}}$ is connected to the \textit{naive local model} ${\rm M}^{\rm naive}_I$ via the \textit{local model diagram} 
 \begin{equation}\label{LMdiagramUnitary}
\mathcal{A}^{\rm naive}_{\mathbf{K}} \ \xleftarrow{\pi^{\rm }_{\mathbf{K}}} \Tilde{\mathcal{A}}^{\rm naive}_{\mathbf{K}} \xrightarrow{q^{\rm }_{\mathbf{K}}} {\rm M}^{\rm naive}_I
\end{equation}
where the morphism $\pi^{\rm }_{\mathbf{K}}$ is a $\mathcal{G}_I$-torsor and $q^{\rm }_{\mathbf{K}}$ is a smooth and $\mathcal{G}_I$-equivariant morphism. 
Equivalently, there exists a relatively representable morphism of algebraic stacks
 \[
\phi:  \mathcal{A}^{\rm naive}_{\mathbf{K}} \to [  {\rm M}_I^{\rm naive}/ \mathcal{G}_I\otimes_{\mathbb{Z}_p}\mathcal{O}]
 \]
which is smooth of relative dimension $\text{dim}(G)$. In particular, since $\calG_I$ is smooth, the above imply that $\mathcal{A}^{\rm naive}_{\mathbf{K}} $ is \'etale locally isomorphic to ${\rm M}^{\rm naive}_I$. Next, we form the cartesian product of $\phi$ with the morphism 
$\Mloc_I \hookrightarrow {\rm M}_I^{\rm naive}$,  where $\Mloc_I  $ is the \textit{local model}, defined as the scheme theoretic closure of the generic fiber $ {\rm M}^{\rm naive}_I \otimes_{\calO} K_1$ in ${\rm M}^{\rm naive}_I$,
\[
\begin{matrix}
\mathcal{A}^{\rm loc}_{\mathbf{K}} &\longrightarrow& [{\rm M}^{\rm loc}_I/ \mathcal{G}_{I,\calO}] \\
\Big\downarrow && \Big\downarrow \\
\mathcal{A}^{\rm naive}_{\mathbf{K}} &\longrightarrow & [{\rm M}^{\rm naive}_I/ \mathcal{G}_{I,\calO}] .
\end{matrix}
\]
The scheme $\mathcal{A}^{\rm loc}_{\mathbf{K}}$ is a closed subscheme of $\mathcal{A}^{\rm naive}_{\mathbf{K}}$ and is a linear modification of $\mathcal{A}^{\rm naive}_{\mathbf{K}}$ in the sense of \cite[\S 2]{P}. 

We now consider a variation of the moduli of abelian schemes $\mathcal{A}^{\rm spl}_{\mathbf{K}}$ where we add in the moduli problem of $\mathcal{A}^{\rm naive}_{\mathbf{K}}$ an additional subspace in the Hodge filtration $ w_{A^\vee} \subset M(A)$ where $M(A)$ is the dual of the deRham cohomology $H_{dR}^1(A)$ of the universal abelian variety $A$ (see \S \ref{ShimuraVarSec} for more details) with certain conditions to imitate the definition of the splitting model ${\rm M}_I^{\rm spl}$. We refer the reader to \S \ref{SplModelsSection} for the definition of the splitting model, which also has an explicit  moduli description. (Here $A^\vee$ is the dual abelian scheme.) There is a forgetful morphism $\tau :   {\rm M}_I^{\rm spl} \longrightarrow  {\rm M}^{\rm naive}_I $ and as above we can form the cartesian product of $\phi $ with $\tau$ and get
\begin{equation}\label{SplCartDiag}
\begin{matrix}
\mathcal{A}^{\rm spl}_{\mathbf{K}} &\longrightarrow & [{\rm M}^{\rm spl}_I/ \mathcal{G}_{I,\calO}] \\
\Big\downarrow && \Big\downarrow \\
\mathcal{A}^{\rm naive}_{\mathbf{K}} &\longrightarrow & [{\rm M}^{\rm naive}_I/ \mathcal{G}_{I,\calO}] .
\end{matrix}
\end{equation}
$\mathcal{A}^{\rm spl}_{\mathbf{K}}$ has the same \'etale local structure as ${\rm M}^{\rm spl}_I$ and it is a linear modification of $\mathcal{A}^{\rm naive}_{\mathbf{K}}$. 
From  now on, we will focus on the non-empty subsets $I  =  \{\ell\}$ where $ 0 \leq \ell \leq m$  and $ \ell \neq 0,\, m-1,\, m$ when $n = 2m$ and $ \ell \neq 0,\, m $ when $ n = 2m + 1$. We will call such index sets \textit{strongly non-special}. In this setting, $P_I$ is a maximal (quasi-)parahoric subgroup. 
In a very recent paper \cite{Luo}, Luo gave a moduli description of $\Mloc_I  $ by adding in ${\rm M}^{\rm naive}_I$ ``the wedge and spin conditions". 
In this paper, 
we consider ${\rm M}^{\rm spl}_I$ and show that by imposing the existence of the additional subspace in ${\rm M}^{\rm spl}_I$ implies the wedge and spin conditions. Thus, we deduce that the scheme theoretic image $\tau({\rm M}^{\rm spl}_I)$ coincides with $\Mloc_I  $.
Notably, the spin conditions are very complicated to check. Therefore, the splitting model that we consider here has a more compact moduli description. Moreover, we want to highlight the novelty in our splitting model's construction: instead of adding two subspaces as expected from \cite[Definition 14.1]{PR2}, we add only one (see Remark \ref{rk 33}). This variation was first used by Richarz \cite{Richarz} in the almost $\pi$-modular case, i.e. $n =2m+1$ is odd and $I={\{m\}}$.

\subsection{} In what follows, we will adhere to the strongly non-special case with signature $(n-1,1)$. One of the main results of the present paper is the following theorem.
\begin{Theorem}\label{IntroRegLM}
     
 For every $K^p$ as above, there is a
 scheme $\mathcal{A}^{\rm spl}_{\mathbf{K}}$, flat over $\Spec(\calO)$, 
 with
 \[
\mathcal{A}^{\rm spl}_{\mathbf{K}}\otimes_{\mathcal{O}} K_1={\rm Sh}_{\mathbf{K}}(G, X)\otimes_{K} K_1,
 \]
 and which supports a local model diagram
  \begin{equation}\label{introLMdiagramReg1}
\begin{tikzcd}
&\wti{\mathcal{A}}^{\rm spl}_{\mathbf{K}}(G, X)\arrow[dl, "\pi^{\rm spl}_K"']\arrow[dr, "q^{\rm spl}_K"]  & \\
\mathcal{A}^{\rm spl}_{\mathbf{K}}  &&  {\rm M}_I^{\rm spl}
\end{tikzcd}
\end{equation}
such that:
\begin{itemize}
\item[a)] $\pi^{\rm spl}_{\mathbf{K}}$ is a $\mathcal{G}_I$-torsor 
and $q^{\rm spl}_{\mathbf{K}}$ is smooth and $\calG_I$-equivariant.

\item[b)] $\mathcal{A}^{\rm spl}_{\mathbf{K}}$ is normal and Cohen-Macaulay and has a reduced special fiber. 
\end{itemize}
\end{Theorem}

From (\ref{SplCartDiag}) and the fact that $  \mathcal{A}^{\rm spl}_{\mathbf{K}}$ is flat we get that $  \mathcal{A}^{\rm loc}_{\mathbf{K}}$ is the image of $\mathcal{A}^{\rm spl}_{\mathbf{K}}$ in $\mathcal{A}^{\rm naive}_{\mathbf{K}}$. From the local model diagram (\ref{introLMdiagramReg1}) it follows that every point of $\mathcal{A}^{\rm spl}_{\mathbf{K}} $ has an \'etale neighborhood which is also \'etale over ${\rm M}^{\rm spl}_I$, and thus it is enough to show that ${\rm M}^{\rm spl}_I$ has the above nice properties. To do this, we explicitly calculate an open affine covering $ \cup_{i_0=1}^n \calU_{i_0}$ of $  \tau^{-1}(*)$ where $*$ is the “worst point” of ${\rm M}^{\rm loc}_I$, i.e. the unique closed $\calG_I$-orbit supported in the special fiber; see \S \ref{Affine charts} 
for more details. The open subschemes $\calU_{i_0} $ are isomorphic to
\[
\Spec \calO[X_1, X_2, Y ]/(\text{rk}(\left[\ 
\begin{matrix}
X_1\\ \hline
X_2  
\end{matrix}\ \right]) - 1,\,\wedge^2\left[\ X_2\mid HY\ \right],\, Y^t\cdot X_2-2\pi)
\]
where $
X_1$, $X_2$ and $Y$ are matrices of sizes $2\ell\times 1$, $(n-2\ell)\times 1$ and $(n-2\ell)\times 1$ with indeterminates as entries and $H$ is the unit antidiagonal matrix of size $(n-2\ell)\times (n-2\ell).$
The rank condition is expressed by imposing that a
certain $i_0$-th entry of the matrix $ [X^t_1\mid X_2^t]$ is a unit. It is worth noting that these affine local charts are very similar to the one that the first author found in the orthogonal local models \cite{I.Z} (see also \cite{PaZa}). By using this observation, we show in \S \ref{FlatReducSec}, that this covering has the nice properties of Theorem \ref{introLMdiagramReg1} (b) and thus we obtain
\begin{Theorem}
\begin{itemize}
\item[a)]
The splitting model ${\rm M}^{\rm spl}_I $ is flat, normal and Cohen-Macaulay. 
\item[b)] The special special fiber of ${\rm M}^{\rm spl}_I $ is reduced and has three irreducible components, which admit moduli descriptions, and one of them is not smooth.
\end{itemize}
\end{Theorem}
Under the local model diagram, (see \S \ref{ShimuraVarSec}), $  \tau^{-1}(*)$ corresponds to the locus where the Hodge filtration $w_{A^\vee}$ of the universal dual abelian scheme $A$ is annihilated by the action of the uniformizer $\pi$. 
Consider the blow-up $\mathcal{A}^{\rm bl}_{\mathbf{K}}$ of $\mathcal{A}^{\rm spl}_{\mathbf{K}} $ along this locus. The second main result of the paper is the following theorem.
\begin{Theorem}
$\mathcal{A}^{\rm bl}_{\mathbf{K}}$ is a semi-stable integral model for the Shimura variety ${\rm Sh}_{\mathbf{K}}(G, X)$.
\end{Theorem}
Since  blowing-up commutes with \'etale localization and the \'etale local structure of the moduli scheme $\mathcal{A}^{\rm spl}_{\mathbf{K}}$ is controlled by the local structure of the local model ${\rm M}_I^{\rm spl}$, it is enough to show the above statement for the corresponding splitting model. By using the explicit open affine covering $ \cup_{i_0=1}^n \calU_{i_0}$ of $  \tau^{-1}(*)$ we show 
\begin{Theorem}
The blow-up ${\rm M}_I^{\rm bl}$ of ${\rm M}_I^{\rm spl}$ along the smooth irreducible component $\tau^{-1}(*) $, which is a projective space of dimension $n-1$ of its special fiber, is regular and has special fiber a divisor with normal crossings. In fact, ${\rm M}_I^{\rm bl}$ is covered by open subschemes which are smooth over $
\Spec (\calO[u, x, y]/(uxy-2\pi))$. 
\end{Theorem} 


We can obtain similar results for the Shimura varieties ${\rm Sh}_{\mathbf{K}'}(G, X)$ where $\mathbf{K}' = K^p P_I^\circ$ (see \S \ref{ShimuraVarSec}). (Recall that $ P_I^\circ$ is the parahoric connected component of the stabilizer $P_I$.) Also, we can apply these results to obtain regular (formal) models of the corresponding Rapoport-Zink spaces.

\subsection{} We conclude this introduction by giving a brief summary of related known results in the literature and discussing some open problems. 
Assume that the signature is $(n-1,1)$. The $\pi$-modular case, when $n =2m$ is even and $I={\{m\}}$, and the almost $\pi$-modular case, when $n =2m+1$ is odd and $I={\{m\}}$, have been studied in \cite[\S 5.3]{PR} and \cite[Prop. 4.16]{A} respectively. These cases are of exotic good reduction, the local models ${\rm M}^{\rm loc}_I$ are smooth and hence are also the corresponding integral models $\mathcal{A}^{\rm loc}_{\mathbf{K}}$. Moreover, using the splitting models we get that in the $\pi$-modular case $\mathcal{A}^{\rm spl}_{\mathbf{K}}$ is smooth \cite[Thm. 1.1]{ZacZhao}, whereas in the almost $\pi$-modular case $\mathcal{A}^{\rm spl}_{\mathbf{K}}$ has semi-stable reduction  \cite[\S 5]{Richarz}. 
For the self-dual case, $I = {0}$, it was shown in \cite[\S 4.5]{P} that ${\rm M}^{\rm loc}_I $ does not have semi-stable reduction but this can be resolved by blowing up the unique singular point of its special fiber. In \cite{Kr}, the author using the corresponding splitting model constructed a regular semi-stable integral model $\mathcal{A}^{\rm spl}_{\mathbf{K}}$ (see also \cite{Zac1}). 
It is worth mentioning that our constructions and proofs of the splitting models generalize to all those cases except for the $\pi$-modular case, where in the moduli problem of the splitting model one needs to add the ``spin condition" to get an honest splitting model; see \cite[Thm. 1.2]{ZacZhao} for more details. Combining the above results with those of this paper, we obtain regular splitting models for any maximal (quasi-)parahoric level.

For general signatures $(n-s,s)$, the results are more sparse: in \cite{BH} and \cite{BZZ} the authors construct flat, Cohen-Macaulay splitting models with reduced special fiber and explicit moduli descriptions for self-dual and  $\pi$-modular level subgroups respectively.  

It is worth mentioning that in many of these works, the authors frequently consider variants of splitting models to obtain an “honest” splitting model, highlighting the need for a more general, group-theoretic definition of splitting models. Finally, we point out that little is known about splitting models for more general (quasi-)parahoric levels, and we intend to pursue this direction in future work.


\section{Preliminaries}\label{Prelim}
\subsection{Pairings and Standard Lattices}\label{Prel.1}
Let $F_0$ be a complete discretely valued field with ring of integers $O_{F_0}$, perfect residue field $k$ of characteristic $\neq 2$, and uniformizer $\pi_0$. We will choose $F/F_0$ to be a ramified quadratic extension and $\pi \in F$ a uniformizer with $\pi^2 = \pi_0$. Consider the $F$-vector space $V$ of dimension $n > 3$ and let 
\[
\phi: V \times V \rightarrow F
\]
be an $F/F_0$-hermitian form which we assume is split, i.e. there is a basis $e_1, \dots, e_n$ of $V$ such that
\[
\phi(ae_i,be_{n+1-j}) = \overline{a}b\delta_{i,j} \quad \text{for  all} \quad a,b \in F, 
\]
where $a \mapsto \overline{a}$ is the nontrivial element of $\text{Gal}(F/{F_0})$. Attached to $\phi$ are the respective alternating and symmetric $F_0$-bilinear forms $V \times V \rightarrow F_0$ given by
\[
\langle x, y \rangle = \frac{1}{2}\text{Tr}_{F /F_0} (\pi^{-1}\phi(x, y)) \quad \text{and} \quad ( x, y ) =  \frac{1}{2}\text{Tr}_{F /F_0} (\phi(x, y)).
\]
For any $O_F$-lattice $\Lambda$ in $V$, we denote by $\hat{\Lambda} = \{v \in V | \langle v, \Lambda \rangle \subset O_{F_0} \}$ the $ \langle \,,\,\rangle $-dual of $\Lambda$ and $\hat{\Lambda}^s$ is related to the $(\, , \,  )$-dual $\hat{\Lambda}^s = \{v \in V | ( v, \Lambda ) \subset O_{F_0} \}$ by the formula $ \hat{\Lambda}^s = \pi^{-1}\hat{\Lambda}$.
Both $\hat{\Lambda}$
and $\hat{\Lambda}^s$ are $O_F$-lattices in V, and the forms $\langle \, , \, \rangle $ and $ (\, , \,) $
induce perfect $O_{F_0}$-bilinear pairings
\begin{equation}\label{perfectpairing}
    \Lambda \times \hat{\Lambda} \xrightarrow{\langle \, , \, \rangle } O_{F_0}, \quad \Lambda \times \hat{\Lambda}^s \xrightarrow{ (\, , \,)} O_{F_0}
\end{equation}
for all $\Lambda$. Also, the uniformizing element $\pi$ induces a $O_{F_0}$-linear mapping on $\Lambda$ which we denote by $t$.

For $i= 0, \dots, n- 1$, we define the standard lattices
\begin{equation}\label{eq 212}
\Lambda_i = \text{span}_{O_F} \{\pi^{-1}e_1, \dots, \pi^{-1}e_i, e_{i+1}, \dots, e_n\}.	
\end{equation}
We consider nonempty subsets $I \subset \{0, \dots, m\}$ where $m = \lfloor n/2\rfloor$ with the requirement that 
\begin{equation}\label{Iproperty}
\text{for } n \text{ even, if } m-1 \in I \Longrightarrow m \in I;
\end{equation}
see \cite[\S 1.2.3(b)]{PR} for more details. (For odd $n$, no condition on $I$ is imposed). We complete the $\Lambda_i$ with $i \in I$ to a self-dual periodic lattice chain $\mathcal{L}_I$ by first including the duals 
$ \Lambda_{n-i}:=\hat{\Lambda}^s_i $ for $i \neq 0$ and then all the $\pi$-multiples: For $j \in \mathbb{Z}$ of the form $j = k \cdot n \pm i$ with $i \in I$ we put $\Lambda_{j} = \pi^{-k}\Lambda_i$. Then $ \{\Lambda_j\}_j$ form a periodic lattice chain $\mathcal{L}_I$ (with $\pi \Lambda_j =  \Lambda_{j-n}$) which satisfies $ \hat{\Lambda}_j = \Lambda_{-j}$. 


\subsection{Unitary Similitude Group and Parahoric Subgroups}\label{ParahoricSbgrs}
Consider the unitary similitude group 
\[
G:=GU(V, \phi) = \{g \in GL_F (V ) \,\, | \,\, \phi(gx, gy) = c(g)\phi(x, y), \quad  c(g) \in F^{\times}_0 \}
\]
and as in \cite[\S 1.1]{PR} choose a partition $n = r + s$ with $s \leq r$. We refer to the pair $(r, s)$ as the signature. 
Identifying $G \otimes F \simeq GL_{n,F} \times \mathbb{G}_{m,F}$, we define the cocharacter $ \mu_{r,s}$ as $ (1^{(s)}, 0^{(r)}, 1)$ of $D \times \mathbb{G}_m$, where $D$ is the standard maximal torus of diagonal matrices in $GL_n$; for more details see \cite[\S 1.1]{PR}. Denote by $E$ the reflex field of $\{ \mu_{r,s}\}$ and by $O := O_E$ its ring of integers. Note that $E = F_0$ if $r = s$ and $E = F$ otherwise (see \cite[\S 1.1]{PR}). 

We next recall the description of the parahoric subgroups of $G$ from \cite[\S 1.2]{PR}. 
We distinguish two cases. When $n=2m+1$ is odd, we let $I$ be a non-empty subset of $\{0,\dots,m\}$ and consider the stabilizer
\[ P_I:=\{g\in G\mid g\Lambda_i=\Lambda_i, \quad \forall i \in I\}\]
of the lattice chain $\Lambda_I$. From \cite[\S 1.2.3(b)]{PR}, we obtain:
\begin{Proposition}
$P_I$ is a parahoric subgroup and every parahoric subgroup of $G$ is conjugate to $P_I$ for a
unique nonempty $I \subset \{0,\dots,m\}$. 
The special maximal parahoric subgroups are exactly those
conjugate to $ P_{\{0\}}$ and $ P_{\{m\}}$.\qed   
\end{Proposition}

When $n=2m$ is even, we let $I$ be a non-empty subset of $\{0,\dots,m\}$ satisfying (\ref{Iproperty}) and we consider the subgroup
\[
P_I:=\{g\in G\mid g\Lambda_i=\Lambda_i , \quad \forall i \in I\}.
\]
$P_I$ is not a parahoric subgroup in
general since it may contain elements with nontrivial Kottwitz invariant. Consider the kernel of the Kottwitz homomorphism:
\[
P_I^0:=\{g\in P_I\mid \kappa(g)=1\}
\]
where $\kappa_G : G(F_0) \twoheadrightarrow \mathbb{Z} \oplus (\mathbb{Z}/2\mathbb{Z})$ (see also \cite[\S 3]{Sm2} for more details). In this case the following statement holds (see \cite[\S 1.2.3(b)]{PR}):

\begin{Proposition}
 $P_I^0$ is a parahoric subgroup and every parahoric subgroup of $G$ is conjugate to $P_I^0$ for a
unique nonempty $I \subset \{0,\dots,m\}$ satisfying (\ref{Iproperty}). For such $I$, we have $P_I^0 = P_I$ exactly when $m \in I$. The special maximal parahoric subgroups are exactly those
conjugate to $P^0_{\{m\}} = P_{\{m\}}$.\qed   
\end{Proposition}
\begin{Definition}\label{DefNonSpecial} {\rm
A non-empty subset $I  =  \{\ell\}$ where $ 0 \leq \ell \leq m$ is called \textit{strongly
non-special} if $ \ell \neq 0,\, m-1,\, m$ when $n = 2m$ and $ \ell \neq 0,\, m $ when $ n = 2m + 1$.}
\end{Definition}

Next, let $I$ be a strongly non-special index set and denote by $P_I$, $ \mathcal{L}_I$ the stabilizer and self-dual multichain defined above for this choice of $I$. Let $\calG_I = \underline{{\rm Aut}}(\mathcal{L}_I)$ be the (smooth) group scheme over $\mathbb{Z}_p$ with $P_I = \mathcal{G}_I(\mathbb{Z}_p)$ the subgroup of $G(\mathbb{Q}_p)$ fixing the lattice chain $\mathcal{L}_I$. Also, the group scheme $\calG_I$ has $G$ as its generic fiber. From the above we deduce that when $n$ is odd, the stabilizer $P_I$ is a parahoric subgroup but when $n$ is even, $P_I$ is not a parahoric subgroup since it contains a parahoric subgroup with index 2. The corresponding parahoric group scheme is its connected component $ P_{I}^\circ$. (See \cite[\S 1.2]{PR} for more details.) 


\section{The unitary moduli problems}\label{LocalModelVariantsSplit}
In this section, we briefly recall the definition of certain variants of local models and splitting models for ramified unitary groups that correspond to the local model triples $(G, \mu_{r,s}, P_I )$ where $I=\{\ell\}$ is a strongly non-special index (see Definition \ref{DefNonSpecial}) and $(r,s) = (n-1,1)$. Note that, given $\dim V=n> 3$ and $s=1<r$, the reflex field $E=F$. 

\subsection{Rapoport-Zink Local Models and Variants}\label{LocalModelVariants}
The \textit{naive local model} of
Rapoport-Zink \cite{RZbook}, $\rmM^{\rm naive}_I$, is the projective scheme over $\Spec O_F$ representing the functor that sends each $O_F$-scheme $S$ to the set of subsheaves 
\[
\calF_\ell \subset  \Lambda_\ell \otimes \mathcal{O}_S, \quad \calF_{n-\ell} \subset  \Lambda_{n-\ell} \otimes \mathcal{O}_S
\]
such that
\begin{enumerate}
    \item $\calF_\ell,\, \calF_{n-\ell}$ as $\mathcal{O}_S $-modules are Zariski locally on $S$ direct summands of rank $n$;
    \item The maps induced by the inclusions $  \Lambda_\ell \subset \Lambda_{n-\ell}  $ and $  \Lambda_{n-\ell} \subset \pi^{-1}\Lambda_\ell $ restrict to maps
    \[
    \calF_\ell \rightarrow \calF_{n-\ell} \rightarrow  \pi^{-1} \calF_{\ell};
    \]
    \item $\calF_{n-\ell}$ is the orthogonal complement of $\calF_\ell$ with respect to $ (\,,\,) : (\Lambda_{n-\ell}\otimes \mathcal{O}_S)\times  (\Lambda_\ell\otimes \mathcal{O}_S) \rightarrow \mathcal{O}_S  $;
    \item (Kottwitz condition) \[\text{char}_{t |  \mathcal{F}_\ell } (X)= (X + \pi)^r(X - \pi)^s \] and the analogous statement holds true for $\calF_{n-\ell}$.
    \end{enumerate}
    
The \textit{wedge local model} ${\rm M}^{\wedge}_I$ is the closed subscheme of ${\rm M}^{\rm naive}_I$ defined by the additional condition that 
    \begin{enumerate}[resume]
        \item (Wedge condition) The exterior powers
\[
 \wedge^{r+1} (t-\pi | \mathcal{F}_\ell) = (0) \quad \text{and} \quad  \wedge^{s+1} (t+\pi | \mathcal{F}_\ell) = (0) 
\]
vanish and the same holds true with $\calF_\ell$ replaced by $\calF_{n-\ell}$.
    \end{enumerate}
The inclusion ${\rm M}^{\wedge}_I \subset {\rm M}^{\rm naive}_I $ is an isomorphism on generic fibers, which both identify naturally with the (smooth) Grassmannian $Gr(s,n)\otimes E$ of dimension $rs$ (see \cite[\S 1.5.3]{PR}). Also, the schemes ${\rm M}^{\wedge}_I$, ${\rm M}^{\rm naive}_I$ support an action of $\mathcal{G}_I$. 

\begin{Remarks}{\rm
In \cite[\S 1.5.1]{PR}, the authors define the naive local model $\rmM^{\rm naive}_I$ that sends each $O_F$-scheme $S$ to the families of $O_F\otimes \calO_S$-modules $(\calF_i\subset \Lambda_i\otimes\calO_S)_{i \in n\mathbb{Z}\pm I}$ that satisfy the conditions (a)-(d) of loc. cit. By periodicity, one can restrict the moduli functor ${\rm M}^{\rm naive}_I$ into the sub-lattices chain $ \Lambda_\ell \subset \Lambda_{n-\ell} \subset \pi^{-1} \Lambda_\ell $ without losing any information. Also, we want to mention that (c) of loc. cit. is equivalent to (3) above.
 }\end{Remarks}

There is a further variant: let $\Mloc_I$ be the scheme theoretic closure of the generic fiber $ {\rm M}^{\rm naive}_I \otimes_{O_F} F$ in ${\rm M}^{\rm naive}_I$. The scheme $\Mloc_I$ is called the \textit{local model}. We have closed immersions of projective schemes
\[
\Mloc_I \subset {\rm M}^{\wedge}_I\subset {\rm M}^{\rm naive}_I
\]
which are equalities on generic fibers (see \cite[\S 1.5]{PR} for more details). Moreover, $ {\rm M}^{\wedge}_I$ is topologically flat
in the strongly non-special case, or in other words, the underlying topological spaces of ${\rm M}^{\wedge}_I$ and ${\rm M}^{\rm loc}_I$ coincide. When $n$ is odd, this follows directly from \cite{Sm1} and when $n$ is even, this follows from \cite[Proposition 7.4.7]{Sm2}.


Moreover, from \cite[Theorem 0.1]{PR}, which in turn uses Zhu’s proof of the Pappas-Rapoport coherence conjecture \cite{Zhu}, we deduce that the special fiber of ${\rm M}^{\rm loc}_I$ is reduced and admits a stratification by Schubert cells indexed by the ``$\mu$-admissible set ” (see also \S \ref{Affine Charts around} for more details). Here, we want to also mention that Zhu \cite{Zhu} uses the powerful technique of Frobenius splitting to prove that the special fiber of the local model is reduced.
\begin{Proposition}
The special fiber of the local model $\Mloc_I$ is reduced, and each irreducible component is normal, Cohen-Macaulay, and Frobenius-split.\qed 
\end{Proposition}
Let us add that, by \cite{HPR}, none of the models $\Mloc_I$ are smooth or semi-stable. Lastly, we want to mention that very recently Luo \cite{Luo} gave a moduli description of $\Mloc_I  $ by adding in ${\rm M}^{\rm naive}_I$ the so-called ``spin conditions"; this will play no role in this paper. 

\subsection{Splitting Models}\label{SplModelsSection}
Next, we consider the moduli scheme ${\rm M}^{\rm spl}_I$ over $O_F$, the \textit{splitting model} as in \cite[\S 5]{Richarz} (see also \cite[Definition 14.1]{PR2}), whose points for an $O_F $-scheme $S$ are the set of triples $(\calF_\ell,\calF_{n-\ell}, \calG_{n-\ell})$ with $ (\calF_\ell,\calF_{n-\ell}) \in {\rm M}^{\wedge}_I(S)$ and 
\[
\calG_{n-\ell} \subset \calF_{n-\ell}
\]
which is Zariski locally $\mathcal{O}_S$-direct summand of rank 1, subject to the following conditions: 
\begin{equation}\label{split.cond.}
(t+\pi) \calF_{n-\ell} \subset \calG_{n-\ell}, \quad  (t-\pi) \calG_{n-\ell} = (0).
\end{equation}
The functor is represented by a projective $O_F $-scheme $ {\rm M}^{\rm spl}_I$. The scheme ${\rm M}^{\rm spl}_I$ supports an action of $\calG_I$ and there is a $\calG_I$-equivariant projective morphism 
\[
 \tau : {\rm M}^{\rm spl}_I \rightarrow {\rm M}^{\wedge}_I \subset {\rm M}^{\rm naive}_I
\]
which is given by $(\calF_\ell,\calF_{n-\ell},\calG_{n-\ell}) \mapsto (\calF_\ell,\calF_{n-\ell})$ on $S$-valued points. (Indeed, we can easily see, as in \cite[Definition 4.1]{Kr}, that $\tau$ is well defined.) The morphism $\tau  :{\rm M}^{\rm spl}_I \rightarrow {\rm M}^{\wedge}_I$ induces an isomorphism on the generic fibers (see \cite[Remark 4.2]{Kr}). 

\begin{Remark}\label{rk 33}{\rm
We want to highlight that in our definition above 
there is only one subspace being added in a dual pair $(\Lambda_i, \hat{\Lambda}^s_i)=(\Lambda_i, \Lambda_{n-i})$, instead of two subspaces as expected from \cite[Definition 14.1]{PR2}. In fact, this variation is necessary because, as one can check, adding two subspaces results in a non-flat model. This variant first appeared in \cite[\S 5.2]{Richarz} where Richarz considered the almost $\pi$-modular case, i.e., when $n=2m+1$ and $I=\{m\}$, and showed that ${\rm M}^{\rm spl}_I$ has semi-stable reduction. 

}\end{Remark}

\begin{Remark}\label{Kott.Wedge}{\rm
In the moduli functor of $ {\rm M}^{\rm spl}_I$ we do not need to impose the Kottwitz condition or the wedge condition. These conditions are implied by condition (\ref{split.cond.}) combined with $(3)$; see \S \ref{EqtsAffChrt} for more details. 
}\end{Remark}

Lastly, in \S \ref{FlatReducSec}, we will prove that ${\rm M}^{\rm spl}_I $ is flat and combining all the above we have that ${\rm M}^{\rm spl}_I \xrightarrow{\tau} {\rm M}^{\wedge}_I$ factors through $\Mloc_I \subset {\rm M}^{\wedge}_I$ because of flatness.

\section{Affine Charts}\label{Affine charts}
\subsection{Affine Charts around $\tau^{-1}(*)$}\label{Affine Charts around}
In what follows, we assume that $I=\{\ell\}$ is a strongly non-special index and $(r,s) = (n-1,1)$.

There is a natural embedding of the special fiber ${\rm M}^{\rm loc}_I\otimes k$ into a partial affine flag variety $\calF l_I$ associated to the unitary group $GU(r,s)$ (see \cite[\S 3]{PR}. Let $\calA_I(\mu)\subset \calF l_I$ be the union of affine Schubert varieties over the $\mu$-admissible set, i.e.,
\[
\calA(\mu)=\cup_{\omega\in {\rm Adm}(\mu)}S_\omega.
\]
By the Pappas-Rapoport coherence conjecture 
we deduce that $\calA(\mu)$ is isomorphic to the special fiber ${\rm M}^{\rm loc}_I\otimes k$ under the natural embedding (see \cite[Prop. 3.1]{PR}, \cite{Zhu}). 
There is a unique closed Schubert cell in ${\rm M}^{\rm loc}_I\otimes k$, which we call the ``worst point", denoted by $*\in {\rm M}^{\rm loc}_I\subset{\rm M}^{\wedge}_I$. In our case, the worst point can be easily represented by the standard lattices $\Lambda_{i}$: For $I=\{\ell\}$, a strongly non-special index, the worst point $*=(t\Lambda_{\ell},t\Lambda_{n-\ell})\in{\rm M}^{\rm loc}_I\otimes k$ (see \cite[Lemma 3.1.1]{Luo} for more details).


To simplify the computation in the following, we reorder the basis of $\Lambda_{\ell}$ (resp. $\Lambda_{n-\ell}$):
\begin{equation}\label{eq 411}
	\Lambda_{\ell}={\rm span}_{O_{F_0}}\{\begin{array}{l}
	e_{n-\ell+1},\cdots, e_n, \pi^{-1}e_1,\cdots, \pi^{-1}e_\ell, e_{\ell+1},\cdots, e_{n-\ell},\\
	\pi e_{n-\ell+1},\cdots, \pi e_n, e_1,\cdots, e_\ell, \pi e_{\ell+1},\cdots, \pi e_{n-\ell}.
\end{array}\},
\end{equation}
\begin{equation*}
\Lambda_{n-\ell}={\rm span}_{O_{F_0}}\{\begin{array}{l}
	e_{n-\ell+1},\cdots, e_n, \pi^{-1}e_1,\cdots, \pi^{-1}e_\ell, \pi^{-1}e_{\ell+1},\cdots, \pi^{-1}e_{n-\ell},\\
	\pi e_{n-\ell+1},\cdots, \pi e_n, e_1,\cdots, e_\ell,  e_{\ell+1},\cdots,  e_{n-\ell}.
\end{array}\}.	
\end{equation*}
By using this new basis order, we have the inverse image 
\[
\tau^{-1}(*)=\{ (t\Lambda_\ell, t\Lambda_{n-\ell},\calG_{n-\ell})\mid {\rm rk}(\calG_{n-\ell})=1, t\calG_{n-\ell}=0\}.
\]
Note that the relation $t\calG_{n-\ell}=0$ is equivalent to $\calG_{n-\ell}\subset t\Lambda_{n-\ell}\otimes k$. Thus, $\tau^{-1}(*)$ contains all $1$-dimensional subspaces in $t\Lambda_{n-\ell}\otimes k$ and is isomorphic to $\PP_k^{n-1}$. It is contained in a union of affine charts $\calU_{i_0}$, i.e.,
\[
\tau^{-1}(*)\subset \cup_{i_0=1}^n \calU_{i_0},
\]
 where $\calU_{i_0}$ is an affine neighborhood around the point $(t\Lambda_\ell, t\Lambda_{n-\ell},k\{\pi e_{n-\ell+i_0}\})$ for $1\leq i_0\leq \ell$ and the point $(t\Lambda_\ell, t\Lambda_{n-\ell},k\{e_{-\ell+i_0}\})$ for $\ell+1\leq i_0\leq n$. In order to find the explicit equations that describe $\U_{i_0}$, we use similar arguments as in the proof of \cite[\S 3]{Zac1} (see also \cite[\S 4]{ZacZhao}). In our case, we ask $(\calF_\ell, \calF_{n-\ell}, \calG_{n-\ell})\in \U_{i_0}$ to satisfy the following conditions:
\[
 \begin{array}{l}
 (1).~ \calF_{\ell}\subset \calF_{n-\ell}\subset t^{-1}\calF_{\ell} ~\text{under the inclusion maps}~ \Lambda_{\ell} \subset \Lambda_{n-\ell}\subset t^{-1}\Lambda_{\ell}.\\
 (2).~  (\calF_{n-\ell},\calF_{\ell})=0 ~\text{with respect to}~ (\ ,\ ): \Lambda_{n-\ell} \times \Lambda_{\ell} \rightarrow \mathcal{O}_S.\\  
 (3).~ \calG_{n-2\ell}\subset \calF_{n-2\ell},\quad (t+\pi) \calF_{n-\ell} \subset \calG_{n-\ell}, \quad  (t-\pi) \calG_{n-\ell} = (0).\\
 (4) \text{(Kottwitz condition)}.~ \text{char}_{t\mid \mathcal{F}_{\kappa} } (T)= (T + \pi)^r(T - \pi)^s \text{ for } \kappa=\ell \text{ or } n-\ell.\\
 (5) \text{(Wedge condition)}.~  \wedge^{r+1} (t-\pi | \mathcal{F}_\kappa) = (0),\, \,  \wedge^{s+1} (t+\pi | \mathcal{F}_\kappa) = (0)  \text{ for } \kappa=\ell, n-\ell.
 \end{array}
 \]
 
 \subsection{Equations of the affine charts $\U_{i_0}$}\label{EqtsAffChrt} A point $(\calF_\ell, \calF_{n-\ell}, \calG_{n-\ell})$ in the affine chart $\U_{i_0}$ can be presented by the following matrices with respect to the basis order (\ref{eq 411}):
 \[  
\calF_\ell = \left[\ 
\begin{matrix}
X\\ \hline
I_n  
\end{matrix}\ \right], \quad 
\mathcal{F}_{n-\ell}=  \left[\ 
\begin{matrix}
Y\\ \hline
I_n  
\end{matrix}\ \right],\quad
\calG_{n-\ell}= \left[\ 
\begin{matrix}
G_1\\ \hline
G_2  
\end{matrix}\ \right],
\]
where matrices $X, Y$ are of sizes $n\times n$ and $G_1, G_2$ are of sizes $n\times 1$. From condition (1) we have the inclusion maps $A_\ell: \Lambda_{\ell}\rightarrow \Lambda_{n-\ell},\,\, A_{n-\ell}: \Lambda_{n-\ell}\rightarrow t^{-1}\Lambda_{\ell}$ which are represented by the following matrices:
\begin{equation}\label{eq 421}
A_\ell=\left[\ 
\begin{matrix}
I_{2\ell} &&&\\ 
&&& \pi_0 I_{n-2\ell}\\
&& I_{2\ell}&\\
&I_{n-2\ell}&& 
\end{matrix}\ \right],\quad	
A_{n-\ell}=\left[\ 
\begin{matrix}
&& \pi_0 I_{2\ell} &\\ 
& I_{n-2\ell}&&\\
I_{2\ell}&&&\\
&&&I_{n-2\ell} 
\end{matrix}\ \right].
\end{equation}
The symmetric pairing $(\ ,\ ): \Lambda_{n-\ell} \times \Lambda_{\ell} \rightarrow \mathcal{O}_S$ is given by
\begin{equation}
(\ ,\ )=\left[\ 
\begin{matrix}
&&J_{2\ell} &\\ 
&&& -H_{n-2\ell}\\
-J_{2\ell}&&&\\
& H_{n-2\ell}&& 
\end{matrix}\ \right],	
\end{equation}
where $H_\ell$ is the unit anti-diagonal matrix of size $\ell$, and $J_{2\ell}=\left[\ 
\begin{matrix}
&H_{\ell} \\ 
-H_{\ell}& 
\end{matrix}\ \right]$. Observe that $ J^2_{2\ell} = -I_{2\ell}$. In the rest of this section, we will omit the lower indices of $H_{\ell}$ and $J_{2\ell}$ for simplicity. 
To find all the equations of $\U_{i_0}$ we claim that it is enough to check the conditions (1)-(3). Observe that condition (2) is equivalent to say that $\calF_{n-\ell}$ is the orthogonal complement of $\calF_{\ell}$. By using the above matrices, it translates to:
\begin{equation}\label{eq 423}
Y=\left[\ 
\begin{matrix}
-J& \\ 
&H 
\end{matrix}\ \right]\cdot X^t\cdot
\left[\ 
\begin{matrix}
J& \\ 
&H 
\end{matrix}\ \right].	
\end{equation}
Condition (3) implies that $\calF_{n-l}$ satisfies the Kottwitz condition and the Wedge condition, i.e., 
\[
\begin{array}{ll}
\det(T\cdot I_n-Y)=(T+\pi)^r(T-\pi)^s,\\
\wedge^{r+1}(Y-\pi I_n)=0,\quad \wedge^{s+1}(Y+\pi I_n)=0.	
\end{array}
\]
Note that $\left[\ 
\begin{matrix}
-J& \\ 
&H 
\end{matrix}\ \right]$ is an invertible matrix, with $\left[\ 
\begin{matrix}
-J& \\ 
&H 
\end{matrix}\ \right]^{-1}=\left[\ 
\begin{matrix}
J& \\ 
&H 
\end{matrix}\ \right]$. From equation (\ref{eq 423}), the characteristic polynomial of $X$ is the same as the characteristic polynomial of $Y$. Similarly, from (\ref{eq 423}) and the invertibility of the above matrix we deduce 
\[
\wedge^{r+1}(X-\pi I_n)=0,\quad \wedge^{s+1}(X+\pi I_n)=0.
\]

We break up the matrices $X, Y$ into blocks as follows:
\begin{equation}\label{eq 424}
X=\left[\ \begin{matrix}
X_1&X_2 \\ 
X_3&X_4 
\end{matrix}\ \right],\quad
Y=\left[\ \begin{matrix}
Y_1&Y_2 \\ 
Y_3&Y_4 
\end{matrix}\ \right],	
\end{equation}
where $X_1$ (resp. $Y_1$) are of sizes $2\ell\times 2\ell$, and $X_4$ (resp. $Y_4$) are of sizes $(n-2\ell)\times (n-2\ell)$. 

\begin{Lemma}\label{lm 41}
Let $H$ be the unit anti-diagonal matrix, and $J=\left[\ 
\begin{matrix}
&H \\ 
-H& 
\end{matrix}\ \right]$.
Let $\calU_{i_0}$ be an affine neighborhood around the point $(t\Lambda_\ell, t\Lambda_{n-\ell},k\{\pi e_{n-\ell+i_0}\})$ for $1\leq i_0\leq \ell$ and the point $(t\Lambda_\ell, t\Lambda_{n-\ell},k\{e_{-\ell+i_0}\})$ for $\ell+1\leq i_0\leq n$.
The affine chart $\U_{i_0}$ is isomorphic to $\Spec O_F[X,Y,V,Z]/I$, where $V, Z$ are the matrices of sizes $n\times 1$, and $I$ is the ideal generated by the following relations:

(a) $Y-VZ^t+\pi I_n$, $Z^tV-2\pi$;

(b) $Y_1+JX_1^tJ$, $Y_2+JX_3^tH$, $Y_3-HX_2^tJ$, $Y_4-HX_4^tH$;

(c) $X_1-(Y_1+Y_2X_3)$, $X_2-Y_2X_4$, $Y_4X_4-\pi I_{n-2\ell}$;

(d) $X_1Y_1-\pi_0 I_{2\ell}$, $Y_3-X_3Y_1$, $Y_4-(X_4+X_3Y_2)$.
\end{Lemma}
	
\begin{proof} From the above discussion, it is enough to check the conditions (1)-(3). 
Note that 
\[
M_t=\left[\ \begin{matrix}
&\pi_0 I_n \\ 
I_n&
\end{matrix}\ \right]
\]
is the matrix giving multiplication by $t$. 
Condition (3) is equivalent to:
\[
\left[\ 
\begin{matrix}
G_1\\ \hline
G_2  
\end{matrix}\ \right]=\left[\ 
\begin{matrix}
Y\\ \hline
I_n  
\end{matrix}\ \right]\cdot G_2;\quad 
\exists Z_{n\times 1}: \left[\ 
\begin{matrix}
\pi Y+\pi_0 I_n\\ \hline
Y+\pi I_n 
\end{matrix}\ \right]=\left[\ 
\begin{matrix}
G_1\\ \hline
G_2
\end{matrix}\ \right]\cdot Z^t; \quad G_1=\pi G_2.
\]	
Set $V=G_2$. Since $G_1=\pi V$ and $Y=VZ^t-\pi I_n$, we get $VZ^tV=2\pi V$ from $G_1=YG_2$. Recall that $\calG_{n-2\ell}$ has rank one. 
From $G_1=\pi G_2$, we deduce that there should be a unit element in 
$G_2=V$. Thus, $VZ^tV=2\pi V$ translates to $Z^tV-2\pi I_n=0$, and therefore we get the relations in (a).

Condition (2) translates to equation (\ref{eq 423}). By using the block matrices from (\ref{eq 424}), we obtain the relations in (b).

Lastly, condition (1) is equivalent to $A_{\ell}\calF_{\ell}\subset\calF_{n-\ell}, \,\, A_{n-\ell}\calF_{n-\ell}\subset t^{-1}\calF_{\ell}$. By using the equation (\ref{eq 421}), we have:
\[
\begin{array}{llll}
X_1=Y_1+Y_2X_3,
& X_2=Y_2X_4,
& Y_3+Y_4X_3=0,
& Y_4X_4=\pi I_{n-2\ell}	,\\
X_1Y_1=\pi_0 I_{2\ell},
&X_1Y_2+X_2=0,
&Y_3=X_3Y_1,
&Y_4=X_4+X_3Y_2.
\end{array}
\]
Using the relations in (b), it is easy to see that equation $Y_3+Y_4X_3=0$ is equivalent to $X_2=Y_2X_4$, and $X_1Y_2+X_2=0$ is equivalent to $Y_3=X_3Y_1$. Thus, we get the relations in (c) and (d). 
\end{proof}

Our next goal is to give a simplification of the ideal $I$ from Lemma \ref{lm 41}. From the proof of the above Lemma, the matrices $X, Y, G_1, G_2$ can be expressed in terms of $V, Z$. Conversely, $V=G_2$, and $Y+\pi I_n=G_2Z^t$ implies that $Z^t$ can be expressed in terms of $Y$, since there exists a unit element in $G_2$.

For later use, we also break up the matrices $V, Z$ into blocks as follows:
\begin{equation}\label{eq 425}
V=\left[\ 
\begin{matrix}
V_1\\ \hline
V_2  
\end{matrix}\ \right],\quad
Z=\left[\ 
\begin{matrix}
Z_1\\ \hline
Z_2  
\end{matrix}\ \right]	
\end{equation}
where $V_1=\left[ v_1\ \cdots\ v_{2\ell} \right]^t$ (resp. $Z_1=\left[ z_1\ \cdots\ z_{2\ell} \right]^t$) and $V_2=\left[ v_{2\ell+1}\ \cdots\ v_{n} \right]^t$ (resp. $Z_2=\left[ z_{2\ell+1}\ \cdots\ z_{n} \right]^t$).

\begin{Proposition}\label{prop 42}
(1) For $1\leq i_0\leq 2\ell$, the affine chart $\U_{i_0}$ is isomorphic to 
\[
\Spec \frac{O_F[V_1, V_2, Z_2]}{(v_{i_0}-1, \,\,\wedge^2\left[\ V_2\mid HZ_2\ \right],\,\, Z_2^t\cdot V_2-2\pi)}.
\]
(2) For $2\ell+1\leq i_0\leq n$, the affine chart $\U_{i_0}$ is isomorphic to 
\[
\Spec \frac{O_F[V_1, V_2, u]}{(v_{i_0}-1,\,\, u\cdot (V_2^tHV_2)-2\pi)}.
\]
\end{Proposition}
\begin{proof}
By Lemma \ref{lm 41} (a) and (\ref{eq 425}), the equation $Y=VZ^t-\pi I_n$ is equivalent to 
\begin{equation}\label{eq 426}
\begin{array}{ll}
Y_1=V_1Z_1^t-\pi I_{2\ell}, & Y_2=V_1Z_2^t,\\
Y_3=V_2Z_1^t, & Y_4=V_2Z_2^t-\pi I_{n-2\ell}.
\end{array}
\end{equation}	
Similarly, $X$ can also be expressed in terms of $V, Z$ by Lemma \ref{lm 41} (b):
\begin{equation}\label{eq 427}
\begin{array}{ll}
X_1=-JZ_1V_1^tJ-\pi I_{2\ell}, & X_2=JZ_1V_2^tH,\\
X_3=-HZ_2V_1^tJ, & X_4=HZ_2V_2^tH-\pi I_{n-2\ell}.
\end{array}
\end{equation}

{\it Case 1:} Assume that $1\leq i_0\leq 2\ell$. Consider the equation $X_1=Y_1+Y_2X_3$ in Lemma \ref{lm 41} (c). From the above, it translates to 
$
JZ_1V_1^tJ+V_1Z_1^t=V_1(Z_2^tHZ_2)V_1^tJ.	
$
Set $a=Z_2^tHZ_2$. The left-hand side of the above relation is $(\pm v_{2\ell+1-j}z_{2\ell+1-i}+v_iz_j)_{i,j}$, where $\pm$ depends on the position of $i, j$ as follows:
\begin{equation}
\pm=\{\begin{array}{ll}
- ~\quad \text{for}~ 1\leq i,j \leq \ell, ~\text{or}~ \ell+1\leq i,j \leq 2\ell,\\
+ ~\quad \text{otherwise.}~ 	
\end{array}	
\end{equation}
Similarly, the right-hand side equals to $(\pm a\cdot v_iv_{2\ell+1-j})_{i,j}$, where $\pm$ is negative if $1\leq j\leq \ell$, and is positive if $\ell+1\leq j\leq 2\ell$. Next, by using this observation, we consider the position $(i,j)=(i_0, 2\ell+1-i_0)$.

When $1\leq i_0\leq \ell$ and $(i,j)=(i_0, 2\ell+1-i_0)$ the above relation gives $2z_{2\ell+1-i_0}=a$. By setting $z_{2\ell+1-i_0}=\frac a2$ and considering the $i_0$-row, we obtain that $Z_1=-\frac{a}{2}JV_1$. When $\ell+1\leq i_0\leq 2\ell$, we get $2z_{2\ell+1-i_0}=-a$. By setting $z_{2\ell+1-i_0}=-\frac a2$, a similar direct calculation shows that $Z_1=-\frac{a}{2}JV_1$. To sum up, we obtain: \vspace{-0.1cm}
\begin{equation}\label{eq 429}
\left[\ z_1 \cdots z_{\ell} ~ z_{\ell+1}\cdots z_{2\ell} \ \right]^t=\frac{a}{2}\left[\ -v_{2\ell} \cdots -v_{\ell+1} ~ v_{\ell}\cdots v_{1}\ \right]^t.	
\end{equation} 
All other positions in $X_1=Y_1+Y_2X_3$ can be obtained by the relation (\ref{eq 429}). Note that relation (\ref{eq 429}) gives $Z_1^tV_1=0$. Thus, equation $Z^tV=2\pi$ in Lemma \ref{lm 41} (a) reduces to $Z_2^tV_2=2\pi$. 

Next, consider the equation $Y_4=X_4+X_3Y_2$ in Lemma \ref{lm 41} (d). By the equations (\ref{eq 426}) and (\ref{eq 427}), we get $V_2Z_2^t=-HZ_2(V_1^tJV_1)Z_2^t+HZ_2V_2^tH$. Note that $V_1^tJV_1=0$ since $J$ is a skew-symmetric matrix. Thus, the above relation translates to $V_2Z_2^t=HZ_2V_2^tH$. By a direct calculation, the following quadratic polynomials $v_{2\ell+i}z_{2\ell+j}-v_{n+1-j}z_{n+1-i}$ vanish for $1\leq i,j \leq n-2\ell$. Thus,
\begin{equation}\label{eq 4210}
\wedge^2\left[\ V_2\mid HZ_2\ \right]=0,~ i.e.,~ \wedge^2\left[ \begin{array}{ccc}
	v_{2\ell+1} &\cdots & v_{n}\\
	z_n &\cdots & z_{2\ell+1}
\end{array} \right]=0.
\end{equation}
It is easy to see that all other relations in Lemma \ref{lm 41} (c), (d) are automatically satisfied by the relations (\ref{eq 429}) and (\ref{eq 4210}). For instance, the relation $X_1Y_1=\pi I_{2\ell}$ translates to $\pi(V_1Z_1^t-JZ_1V_1^tJ)=0$, which is satisfied since $Z_1=-\frac{a}{2}JV_1$. Relation $X_2=Y_2X_4$ is equivalent to $aV_2^t=2\pi Z_2^tH$. Recall that $a=Z_2^tHZ_2$. By $V_2Z_2^t=HZ_2V_2^tH$, we get $Z_2^tV_2Z_2^tH=2\pi Z_2^tH$. It is satisfied since $Z_2^tV_2=2\pi$. Similarly, relation $Y_4X_4=\pi_0 I_{n-2\ell}$ is equivalent to $V_2(aV_2^t-2\pi Z_2^tH)=0$, and relation $Y_3=X_3Y_1$ is equivalent to $JV_1(aV_2^t-2\pi Z_2^tH)=0$. Therefore, by all these simplifications we get that $\U_{i_0}$ is isomorphic to the spectrum of the quotient ring
\begin{equation}
 \frac{O_F[V_1, V_2, Z_2]}{(v_{i_0}-1, \,\,\wedge^2\left[\ V_2\mid HZ_2\ \right],\,\, Z_2^t\cdot V_2-2\pi)}.
\end{equation}

{\it Case 2:} Assume that $2\ell+1\leq i_0\leq n$. Consider the relation $Y_4=X_4+X_3Y_2$. From the above, it is equivalent to $\wedge^2\left[\ V_2\mid HZ_2\ \right]=0$. Note that in this case $v_{i_0}=1$ and $ v_{i_0}$ is an entry from $V_2$. Using this and the minor relations we get:
\begin{equation}\label{eq 4212}
\left[\ z_{2\ell+1} \cdots z_{n}\ \right]^t=u\cdot \left[\ v_{n} \cdots v_{2\ell+1}\ \right]^t, \quad\text{where}~ u=z_{n+2\ell+1-i_0}.
\end{equation}
Thus, $Z_2=u\cdot HV_2$. The relation $X_2=Y_2X_4$, from Lemma \ref{lm 41} (c), translates to $(JZ_1-(a-\pi u)V_1)V_2^t=0$ by using (\ref{eq 4212}). Since $ v_{i_0}$ is an entry from $V_2$ with $v_{i_0}=1$ we deduce
\begin{equation}\label{eq 4213}
Z_1=-(a-\pi u)JV_1.
\end{equation}
So far, the matrices $Z_1, Z_2$ are expressed in terms of $V_1, V_2, a, u$. By $Z^tV=2\pi$, from Lemma \ref{lm 41} (a), and (\ref{eq 4213}) we get $Z_2^tV_2=2\pi$, which in turn translates to $u\cdot (V_2^tHV_2)=2\pi$ since $Z_2=u\cdot HV_2$. Note that, in this case, $a=Z_2^tHZ_2=2\pi u$ and so $Z_1=-\pi u\cdot JV_1=-\frac{a}{2}JV_1$ (which is the same as relation (\ref{eq 429})).

All the other relations from the ideal $I$ are automatically satisfied by using the above relations. Thus, $\U_{i_0}$ is isomorphic to the spectrum of the quotient ring:
\begin{equation}
\frac{O_F[V_1, V_2, u]}{(v_{i_0}-1,\,\, u\cdot (V_2^tHV_2)-2\pi)}.
\end{equation} 
\end{proof}
\begin{Remarks}\label{rk 43}
{\rm (a) As we can see from the above proof of {\it Case 2},  the relation $\wedge^2\left[\ V_2\mid HZ_2\ \right]=0$ translates to $Z_2=u\cdot HV_2$ and $Z_2^tV_2=2\pi$ is equivalent to $u\cdot (V_2^tHV_2)=2\pi$. Thus, it is not very hard to observe that for $1\leq i_0\leq n$:
\begin{equation}
\U_{i_0}\simeq \Spec \frac{O_F[V_1, V_2, Z_2]}{(v_{i_0}-1,\,\, \wedge^2\left[\ V_2\mid HZ_2\ \right], \,\,Z_2^t\cdot V_2-2\pi)}.
\end{equation}

(b) From Proposition \ref{prop 42} (2), the special fiber of $\U_{i_0}$ has two irreducible components when $2\ell+1\leq i_0\leq n$, i.e.,
\begin{equation}
V(I)=V(u)\cup V(V_2^tHV_2).
\end{equation}
Both of them are smooth of dimension $n-1$ and their intersection is smooth of dimension $n-2$. Thus, in this case, the affine chart $\U_{i_0}$ has semi-stable reduction over $O_F$, i.e., it is regular and its special fiber is reduced with normal crossings. However, when $1\leq i_0\leq 2\ell$, $\U_{i_0}$ does not have semi-stable reduction; see Remark \ref{NotNormalCrossing}. We will study this case in detail in the next section.
}\end{Remarks}



\section{Flatness and Reducedness of the splitting model ${\rm M}^{\rm spl}_I $}\label{FlatReducSec}
In all of Section \ref{FlatReducSec}, we assume $I=\{\ell\}$ is a strongly non-special index (see Definition \ref{DefNonSpecial}) and $(r,s) = (n-1,1)$. 
Our goal in this section is to show that ${\rm M}^{\rm spl}_I $ is flat over $\Spec O_F$ and its special fiber is reduced. In the course of proving these we will also deduce that ${\rm M}^{\rm spl}_I $ is normal and Cohen-Macaulay. 

\begin{Theorem}\label{SplitFlat}
\begin{itemize}
\item[a)] The splitting model ${\rm M}^{\rm spl}_I $ is $O_F$-flat, normal and Cohen-Macaulay.
\item[b)] The special fiber of ${\rm M}^{\rm spl}_I $ is reduced.
\end{itemize}
\end{Theorem}

\begin{proof}
From the construction of splitting models and \S \ref{Affine charts}, it's enough to show that the open subschemes
\[
U_{i_0} = \Spec \frac{O_F[V_1, V_2, Z_2]}{(v_{i_0}-1, \,\wedge^2\left[\ V_2\mid HZ_2\ \right],\, Z_2^t\cdot V_2-2\pi)}
\]
for $1 \leq i_0 \leq n$, are flat, normal, Cohen-Macaulay and with reduced special fiber. By Remark \ref{rk 43}, when $ 2\ell+1\leq i_0 \leq n $, $U_{i_0}$ is regular and its special fiber is reduced with normal crossings. So, it's enough to consider the case $ 1\leq i_0 \leq 2\ell.$ In this case, we have that $U_{i_0} \cong \mathbb{A}^{2\ell-1}_{O_F} \times T$ where
\[
T = \Spec O_F[V_2,Z_2]/
(\wedge^2\left[\ V_2\mid HZ_2\ \right], Z_2^t\cdot V_2-2\pi).
\]
We have the isomorphism 
\[
\frac{O_F[V_2,Z_2]}{ 
(\wedge^2\left[\ V_2\mid HZ_2\ \right],\, \,Z_2^t\cdot V_2-2\pi)} \cong   \frac{O_F[V_2,Z_2]}{ 
(\wedge^2\left[\ V_2\mid Z_2\ \right],\,\, Z_2^t\cdot H \cdot V_2-2\pi)}
\]
given by $ V_2 \mapsto V_2$, $Z_2 \mapsto H\cdot Z_2 $. Also, by setting $ \mathcal{Z}_1 : = V^t_2 $, $ \mathcal{Z}_2 : = Z^t_2 $ where $\mathcal{Z} = \left[\ 
\begin{matrix}
\mathcal{Z}_1\\ \hline
\mathcal{Z}_2  
\end{matrix}\ \right] = (z_{ij})\in {\rm Mat}_{2 \times (n-2\ell)}$ we get that the scheme $T$ is isomorphic to 
\[
\Spec O_F[\mathcal{Z}]/( \wedge^2\mathcal{Z},\,  \sum^{n-2\ell}_{ i=1}z_{1,i}\, z_{2,n-2\ell+1-i}-2\pi )
\]
and its special fiber $\overline{T}$  is isomorphic to
\[
\Spec k[\mathcal{Z}]/( \wedge^2\mathcal{Z},\, \sum^{n-2\ell}_{ i=1}z_{1,i}\, z_{2,n-2\ell+1-i} ).
\]
This scheme has already been studied at \cite{I.Z} (see also \cite[Theorem 5.1.1]{PaZa}) in a more general setting  where the size of the matrix $\mathcal{Z}$ was $n \times m$. In particular, 
from \cite[\S 5, \S 8]{I.Z} we have that $T$ is $O_F$-flat, Cohen-Macaulay and of relative dimension $n-2\ell$. Also, from \cite[\S 6, \S 8]{I.Z} we get that $\overline{T}$ is reduced. Since, $T$ is $O_F$-flat with smooth generic fiber (see \S \ref{LocalModelVariantsSplit}) and reduced special fiber, it follows by \cite[Proposition 9.2]{PZ} that $T$ is normal. In the course of proving the reducedness of the special fiber in \cite{I.Z}, the first author also determines its irreducible components (see \cite[\S 7.1, \S 9.2.2 ]{I.Z}). Translating his results in our setting we obtain that the irreducible components of $\overline{T} $ are $ V(J_1)$, $V(J_2)$, $V(J_3)$ where
\[
 J_1 = (Z_2),   \quad J_2  = \left(  \wedge^2\left[\ V_2\mid HZ_2\ \right], \,\,  Z_2^t\cdot V_2, \,\, V_2^t\cdot H \cdot V_2,  \,\,Z_2^t\cdot H \cdot Z_2\right), \quad J_3 = (V_2).
\]
From \cite[\S 7]{I.Z}, we obtain that $V(J_2)$ has dimension $n-2\ell$ and it is smooth over $\Spec(k)$ outside from its closed subscheme of dimension 0 that is defined by the ideal $\left(V_2,\,Z_2\right)$. Also, we can easily see that $V(J_1)$, $V(J_3)$ are smooth affine spaces of dimension $n-2\ell$. 
\end{proof}
\begin{Remark}\label{NotNormalCrossing}
    \rm{ From the above proof and more precisely from the fact that the irreducible component $V(J_2)$ is singular, we conclude that the splitting model ${\rm M}^{\rm spl}_I $ is not semi-stable over $\Spec O_F$. }
\end{Remark}
\begin{Corollary}\label{ProjMor}
 The projective morphism $\tau: {\rm M}^{\rm spl}_I\rightarrow {\rm M}^{\wedge}_{I}$ factors through ${\rm M}^{\rm loc}_I$.
\end{Corollary}
\section{Moduli description of the irreducible components of ${\rm M}^{\rm spl}_I\otimes k$}
In this section, we give a moduli description of the irreducible components of ${\rm M}^{\rm spl}_I\otimes k$ when the signature is $(n-1,1)$ and the index set $I=\{\ell\}$ is strongly non-special. Recall that we have a symmetric bilinear form
\begin{equation}
	(\ ,\ ):(\Lambda_{n-\ell}\otimes \calO_S)\times (\Lambda_{\ell}\otimes \calO_S)\rightarrow \calO_S.
\end{equation} \vspace{-0.1cm}
For simplicity, we omit the base change from the above notation. Set $(t\pm \pi)\Lambda_{\kappa}=\{(t\pm \pi)x\mid x\in \Lambda_\kappa\}$. It is easy to see that 
\begin{equation}\label{eq 522}
	((t+ \pi)\Lambda_{n-\ell})^\perp=(t+ \pi)\Lambda_{\ell},\quad ((t- \pi)\Lambda_{n-\ell})^\perp=(t- \pi)\Lambda_{\ell}.
\end{equation}
By (\ref{eq 522}), we have a perfect pairing between $(t+ \pi)\Lambda_{n-\ell}$ and $\Lambda_{\ell}/(t+ \pi)\Lambda_{\ell}$. Note that the last lattice is isomorphic to $(t- \pi)\Lambda_{\ell}$ via $e_i\mapsto (t-\pi)e_i$. Thus, we have an induced perfect pairing 
\begin{equation}
\{\ ,\ \}:(t+ \pi)\Lambda_{n-\ell}\times (t- \pi)\Lambda_{\ell}\rightarrow \calO_S,
\end{equation}
given by $\{(t+ \pi)x, (t- \pi)y\}=((t+ \pi)x, y)$. For points $(\calF_\ell, \calF_{n-\ell}, \calG_{n-\ell})$ in the splitting model ${\rm M}^{\rm spl}_I$, we have $\calG_{n-\ell}\subset (t+ \pi)\Lambda_{n-\ell}$ since $(t-\pi)\calG_{n-\ell}=(0)$. Thus, we can consider the orthogonal complement of $\calG_{n-\ell}$ under this new modified pairing, which we denote by $\calG_{n-\ell}^{\perp'}\subset (t- \pi)\Lambda_{\ell}$. 

From the inclusion maps $A_\ell: \Lambda_{\ell}\rightarrow \Lambda_{n-\ell}, A_{n-\ell}: \Lambda_{n-\ell}\rightarrow t^{-1}\Lambda_{\ell}$, we define:

\begin{Definition}
Let $M_i$ be the closed subscheme in the special fiber of ${\rm M}^{\rm spl}_I$, which send each $(O_F\otimes k)$-algebra $R$ to the set of families:
\[
\begin{array}{l}
M_1(R)=\{(\calF_\ell, \calF_{n-\ell}, \calG_{n-\ell})\in {\rm (M}^{\rm spl}_I\otimes k)(R)\mid t\calF_{\ell}=0,\quad t\calF_{n-\ell}=0\},\\
M_2(R)=\{(\calF_\ell, \calF_{n-\ell}, \calG_{n-\ell})\in {\rm (M}^{\rm spl}_I\otimes k)(R)\mid tA_{\ell}^{-1}(\calG_{n-\ell})\subset \calG_{n-\ell}^{\perp'}, A_{\ell}(\calF_{\ell})\subset t\Lambda_{n-\ell}     \},\\
M_3(R)=\{(\calF_\ell, \calF_{n-\ell}, \calG_{n-\ell})\in {\rm (M}^{\rm spl}_I\otimes k)(R)\mid \calG_{n-\ell}\subset A_\ell(\calG_{n-\ell}^{\perp'})\}.\\	
\end{array}
\]
\end{Definition}

From the above definition, we get

\begin{Theorem}
The special fiber of the splitting model ${\rm M}^{\rm spl}_I$ has three irreducible components, $M_1, M_2$ and $M_3$.
\end{Theorem}
\begin{proof}
By passing to the affine chart $\U_{i_0}$ of the inverse image of the worst point as in \S \ref{Affine charts} and the proof of Theorem \ref{SplitFlat}, we only need to identify the three  irreducible components $V(J_1), V(J_2), V(J_3)$ with  $M_1, M_2, M_3$ respectively.

Choose a point $(\calF_\ell, \calF_{n-\ell}, \calG_{n-\ell})\in \U_{i_0}$ as \S \ref{Affine Charts around}:
\[  
\calF_\ell = \left[\ 
\begin{matrix}
X\\ \hline
I_n  
\end{matrix}\ \right], \quad 
\mathcal{F}_{n-\ell}=  \left[\ 
\begin{matrix}
Y\\ \hline
I_n  
\end{matrix}\ \right],\quad
\calG_{n-\ell}= \left[\ 
\begin{matrix}
G_1\\ \hline
G_2  
\end{matrix}\ \right].
\]
It is easy to see that relation $t\calF_\ell=0$ is equivalent to $X=0$. Note that $X=0$ if and only if $Y=0$ by $(\calF_{n-\ell}, \calF_{\ell})=0$, and $Y=0$ is equivalent to $t\calF_{n-\ell}=0$. In the special fiber, we have $Y=VZ^t$. Recall that there is a non-vanishing element in $V$ by rk$(\calG_{n-\ell})=1$. Thus, the relation $Y=0$ translates to $Z=0$, which amounts to $Z_2=0$ by $Z_1=-\frac{a}{2}JV_1, a=Z_2^tHZ_2$. Therefore, the relations $t\calF_\ell=0, t\calF_{n-\ell}=0$ are equivalent to $Z_2=0$, where the ideal $(Z_2)$ is $J_1$ by Theorem \ref{SplitFlat}.

For $M_3(R)$, consider the orthogonal complement $\calG_{n-\ell}^{\perp'}$. Note that the modified pairing $\{\ ,\ \}$ under the reordered basis of $t\Lambda_{\ell}, t\Lambda_{n-\ell}$ in (\ref{eq 411}) is
\begin{equation}
\{\ ,\ \}=\left[\
\begin{array}{cc}
	-J_{2\ell} &\\
	&H_{n-2\ell}	
\end{array}\
\right].
\end{equation}
From the above, we get 
\begin{equation}\label{eq 525}
\calG_{n-\ell}^{\perp'}=\{\left[\
\begin{matrix}
W_1\\ \hline
W_2  
\end{matrix}\ \right]\mid -V_1^tJW_1+V_2^tHW_2=0
\}\subset
t\Lambda_{\ell},	
\end{equation}
 where $W_1$ (resp. $W_2$) is of size $2\ell\times 1$ (resp. $(n-2\ell)\times 1$). Note that rk$(\calG_{n-\ell}^{\perp'})=n-1$. The image of $\calG_{n-\ell}^{\perp'}$ under $A_\ell$ is
 \begin{equation}
 	A_\ell(\calG_{n-\ell}^{\perp'})=
 	\{\left[\
\begin{matrix}
W_1\\ \hline
0  
\end{matrix}\ \right]\mid \exists W'_{(n-2\ell)\times 1} ~\text{such that}-V_1^tJW_1+V_2^tHW'=0
\}.
 \end{equation} 
Thus, the relation $\calG_{n-\ell}\subset A_\ell(\calG_{n-\ell}^{\perp'})$ is equivalent to $V_2=0$, since $V_1^tJV_1=0$. We deduce that $V(J_3)=V((V_2))$ represents $M_3(R)$.

Lastly, note that 
\begin{equation}
tA_{\ell}^{-1}(\calG_{n-\ell})=\text{span}_{k}\{
\left[\
\begin{matrix}
0\\ \hline
V_2  
\end{matrix}\ \right]
\}\subset t\Lambda_\ell.
\end{equation}
By (\ref{eq 525}), the relation $tA_{\ell}^{-1}(\calG_{n-\ell})\subset \calG_{n-\ell}^{\perp'}$ is equivalent to $V_2^tHV_2=0$. It is easy to see that $A_\ell(\calF_\ell)\subset t\Lambda_{n-\ell}$ translates to $X_1=0, X_2=0$. When $2\ell+1\leq i_0\leq n$, it is automatically satisfied since $X_1=-\frac{a}{2} V_1V_1^tJ, \,\,X_2=\frac{a}{2}V_1V_2^tH$, and $a=2\pi u=0$ in the special fiber. So it is enough to consider $1\leq i_0\leq 2\ell$, where the relation $X_1=0, X_2=0$ is equivalent to $a=Z_2^tHZ_2=0$ by the proof of Proposition \ref{prop 42}. Thus, we get $V(J_2)$ represents $M_2(R)$.
\end{proof}

\section{A resolution for ${\rm M}^{\rm spl}_I $}\label{ResolSpl}
\vspace{-0.1cm}
In what follows, we continue to assume that $I=\{\ell\}$ is a strongly non-special index and the signature $(r,s) = (n-1,1)$. 

While ${\rm M}^{\rm spl}_I$ is flat over $O_F$ from Theorem \ref{SplitFlat}, it does not have semi-stable reduction since one of the irreducible components of ${\rm M}^{\rm spl}_I\otimes k$ is not smooth; see Remark \ref{NotNormalCrossing}. In this section, we will construct a semi-stable resolution $\rho: {\rm M}_I^{\rm bl}\rightarrow {\rm M}^{\rm spl}_I$ and show that ${\rm M}_I^{\rm bl}$ is the blow-up of ${\rm M}^{\rm spl}_I$ along the closed subscheme $\tau^{-1}((t\Lambda_{\ell}, t\Lambda_{n-\ell}))$. 
\vspace{-0.1cm}
\subsection{Blow-up of $\U_{i_0}$}\label{Blow-up} From Proposition \ref{prop 42} and Remarks \ref{rk 43} (a), we have that $\U_{i_0}\simeq \Spec A=\Spec O_F [V_1, V_2, Z_2]/I$, where $I$ is given by:
\begin{equation}
I=	(v_{i_0}-1, \,\,\wedge^2\left[\ V_2\mid HZ_2\ \right],\,\, Z_2^t\cdot V_2-2\pi)
\end{equation}
for $1\leq i_0\leq n$. The smooth closed subscheme $U_{i_0}\cap \tau^{-1}((t\Lambda_{\ell}, t\Lambda_{n-\ell}))$ is defined by the ideal $I'=(Z_2)$. Let $\U_{i_0}^{\rm bl}$ be the blow-up of $\U_{i_0}$ along the ideal $I'=(Z_2)$ and $\rho: \U_{i_0}^{\rm bl}\rightarrow \U_{i_0}$ the blow-up morphism. We have:
\begin{equation}
	\U_{i_0}^{\rm bl}={\rm Proj (\tilde{A})},  
\end{equation}
where $\tilde{A}$ is the graded $A$-algebra $\oplus_{d\geq 0} I^d$ with $I^0=A$.

\begin{Proposition}\label{prop 61}
$\U_{i_0}^{\rm bl}$ has semi-stable reduction over $O_F$ for $1\leq i_0\leq n$. 	
\end{Proposition}
\vspace{-0.1cm}
\begin{proof}
When $1\leq i_0\leq 2\ell$, we get
\begin{equation}
\U_{i_0}=\Spec A=\mathbb{A}_{O_F}^{2\ell-1}\times \Spec \frac{O_F[V_2, Z_2]}{(\wedge^2\left[\ V_2\mid HZ_2\ \right], \,\,Z_2^t\cdot V_2-2\pi)}.	
\end{equation}
There are $(n-2\ell)$- affine patches in the blow-up of $\Spec A$ along the ideal $I'=(Z_2)=(z_{2\ell+1}, \cdots, z_{n})$. Let $t_{2\ell+1},\cdots, t_n$ be the represented elements of $z_{2\ell+1}, \cdots, z_{n}$ in the graded algebra $\tilde{A}$, considered as homogeneous elements of degree $1$. 

Assume that $t_{2\ell+1}=1$. The affine chart $D_+(t_{2\ell+1})\subset \U_{i_0}^{\rm bl}$ is  given by 
\begin{equation}
\Spec \frac{A[t_{2\ell+2},\cdots, t_{n}]}{J_{2\ell+1}},	
\end{equation}
where 
\begin{equation}\label{eq 515}
J_{2\ell+1}=\{f(t_{2\ell+2},\cdots ,t_{n})\in A[t_{2\ell+2},\cdots ,t_{n}]\mid \exists N\geq 0, z_{2\ell+1}^N\cdot f\in (z_{2\ell+1}t_{i}-z_i)_{2\ell+1\leq i\leq n}\}.	
\end{equation}
From (\ref{eq 515}), it is easy to see that $v_i-t_{n+2\ell+1-i}v_n\in J_{2\ell+1}$. Thus, the affine chart $D_+(t_{2\ell+1})$ is isomorphic to 
\begin{equation}
\mathbb{A}_{O_F}^{2\ell-1}\times \Spec \frac{O_F[v_n, z_{2\ell+1}, t_{2\ell+1},\cdots, t_{n}]}{(t_{2\ell+1}-1,\,\,v_nz_{2\ell+1}\cdot q-2\pi)}	
\end{equation}
where $q=\sum_{i=2\ell+1}^n t_it_{n+2\ell+1-i}$. The generic fiber of $D_+(t_{2\ell+1})$ is smooth. In the special fiber, we have three irreducible components, corresponding to $J_1=(v_n), J_2=(z_{2\ell+1})$, and $J_3=(q)$. All of them are isomorphic to $\mathbb{A}_k^{n-1}$, and their intersections are smooth of correct dimension. Next, note that the special fiber of $V(J_{2\ell+1})$ is reduced, since $(v_nz_{2\ell+1} q)=J_1\cap J_2\cap J_3$. We can directly see that the ideal $J_1, J_2, J_3$ are principal over $V(J_{2\ell+1})$. Thus by \cite[Remark 1.1.1]{Ha}, $V(J_{2\ell+1})$ is regular, and has semi-stable reduction, which in turn gives that $D_+(t_{2\ell+1})$ has semi-stable reduction.

By symmetry, all other affine patches $D_+(t_i)$ in $\U_{i_0}^{\rm bl}$ are isomorphic to  $D_+(t_{2\ell+1})$. Therefore, $\U_{i_0}^{\rm bl}$ has semi-stable reduction.

When $2\ell+1\leq i_0\leq n$, we get 
\begin{equation}
\U_{i_0}\simeq \Spec \frac{O_F[V_1, V_2, u]}{(v_{i_0}-1,\,\, u\cdot (V_2^tHV_2)-2\pi)}	
\end{equation}
by Proposition \ref{prop 42} (2). Note that the ideal $I'=(Z_2)$ translates to $I'=(u)$ in this case. Thus, we get $\U_{i_0}^{\rm bl}=\U_{i_0}$ and so it has semi-stable reduction by Remark \ref{rk 43} (b).
\end{proof}

Set $T=\left[\ t_{2\ell+1} \cdots t_n \ \right]^t$. As a consequence of the proposition above, we have:

\begin{Corollary}\label{Cor.72}
 The blow-up $\U_{i_0}^{\rm bl}$ is isomorphic to 
	\begin{equation}
	{\rm Proj} \frac{B[T]}{(\wedge^2 \left[\ Z_2\mid T \ \right], \,\,\wedge^2 \left[\ V_2\mid HT \ \right])},
	\end{equation}
	for $1\leq i_0\leq n$ where $ B=O_F [V_1, V_2, Z_2]/(v_{i_0}-1, Z_2^t\cdot V_2-2\pi)$.
\end{Corollary}

\subsection{A resolution for ${\rm M}^{\rm spl}_I $} Recall that $\tau^{-1}(t\Lambda_\ell, t\Lambda_{n-\ell})$ is the smooth $\calG_I$-invariant closed subscheme, which is isomorphic to $\mathbb{P}_k^{n-1}$ in the special fiber of the splitting model ${\rm M}^{\rm spl}_I $, and $\tau^{-1}(t\Lambda_\ell, t\Lambda_{n-\ell})\subset \cup_{i_0=1}^n \U_{i_0}$. Consider the blow-up of ${\rm M}^{\rm spl}_I $ along the subscheme $\tau^{-1}(t\Lambda_\ell, t\Lambda_{n-\ell})$. This gives a $\calG_I$-equivariant, birational projective morphism:
\begin{equation}
	\rho: {\rm M}_I^{\rm bl} \rightarrow {\rm M}^{\rm spl}_I,
\end{equation}
which induces an isomorphism on the generic fibers.

\begin{Theorem}\label{ResolutionSpl}
The scheme ${\rm M}_I^{\rm bl}$ is regular, and has semi-stable reduction over $O_F$.
\end{Theorem}

\begin{proof}
Since $\tau^{-1}(t\Lambda_\ell, t\Lambda_{n-\ell})\subset \cup_{i_0=1}^n \U_{i_0}$, it is enough to check the restriction of $\rho^{-1}(\cup_{i_0=1}^n\U_{i_0})$. From \S \ref{Blow-up}, the blow-up of $\U_{i_0}$ along $\U_{i_0}\cap \tau^{-1}((t\Lambda_{\ell}, t\Lambda_{n-\ell}))$ is $\U_{i_0}^{\rm bl}$, i.e., we have $\rho^{-1}(\U_{i_0})=\U_{i_0}^{\rm bl}$. We obtain a $\calG_I$-equivalent, birational, semi-stable resolution 
\begin{equation}
	\rho: \U_{i_0}^{\rm bl}\rightarrow \U_{i_0}
\end{equation}
 by Proposition \ref{prop 61}. Thus, $\rho^{-1}(\cup_{i_0=1}^n\U_{i_0})$ is also regular, and has semi-stable reduction over $O_F$ and so is ${\rm M}^{\rm bl}_I$.
\end{proof}
The semi-stable reduction of ${\rm M}^{\rm bl}_I$ implies that it is flat over $O_F$. Combining all the above, we have the $\calG_I$-equivariant projective morphisms
\begin{equation}\label{eq.721}
{\rm M}^{\rm bl}_I\rightarrow
{\rm M}^{\rm spl}_I\rightarrow	
{\rm M}^{\rm loc}_I,
\end{equation} 
which induces isomorphisms on the generic fibers. Moreover, from Theorem \ref{ResolutionSpl} and the proof of Proposition \ref{prop 61}, we deduce that ${\rm M}_I^{\rm bl}$ is covered by open subschemes which are smooth over $\Spec (O_F[u, x, y]/(uxy-2\pi))$.
\begin{Remark}\label{NotNormalCrossing1}
    \rm{ From Subsection \ref{Blow-up} and Theorem \ref{ResolutionSpl}, we deduce that the special fiber of ${\rm M}_I^{\rm bl}$ has three smooth irreducible components of dimension $n-1$. Also, the exceptional locus of (\ref{eq.721}) is a closed subscheme of $ \mathbb{P}^{n-1}_k\times_k \mathbb{P}^{n-2\ell-1}_k$ which is isomorphic to $ V(I')$ over the intersection with $\U_{i_0}^{\rm bl}$; see the proof of Proposition \ref{prop 61} and Corollary \ref{Cor.72} for more details. Thus, the exceptional locus is a smooth scheme of dimension $n-1$ and is, in fact, an irreducible component of the special fiber of ${\rm M}^{\rm bl}_I$.}
\end{Remark}

\section{Application to Shimura varieties}\label{ShimuraVarSec}
In this section, we demonstrate the most immediate application of Theorems \ref{SplitFlat} and \ref{ResolutionSpl} to ramified unitary Shimura varieties of signature $(n-1,1)$ and where the level subgroup is the stabilizer of a lattice which is neither special nor self-dual. We start by briefly discussing the construction of $p$-adic integral models of Shimura varieties, defined in \cite{RZbook}, which have simple and explicit moduli descriptions and are \'etale locally isomorphic to naive local models. We follow \cite[\S 1]{PR} for the description.

Let $K/\mathbb{Q}$ be an imaginary quadratic extension. Let $W$ be a $n$-dimensional $K$-vector space, equipped with a non-degenerate hermitian form $\phi$. Consider the group $G = GU_n$ of unitary similitudes for $(W,\phi)$ of  dimension $n> 3$ over $K$. Fix a conjugacy class of homomorphisms $h: \Res_{\mathbb{C}/\mathbb{R}}\mathbb{G}_{m,\mathbb{C}}\rightarrow GU_n$ that corresponds to a Shimura datum $(G,X) = (GU_n,X_h)$ of signature $(n-1,1)$. The pair $(G,X)$ gives rise to a Shimura variety $Sh(G,X)$ over the reflex field $E=K$. Let $p$ be an odd prime number which ramifies in $K$ and set $K_1 =K_v $ where $v$ is the unique prime ideal of $K$ above $(p)$. Denote by $\mathcal{O}$ the ring of integers of $K_1$ and let $\pi$ be a uniformizer of $\mathcal{O}$. Set $V=W\otimes_{\QQ} \QQ_p$. 
We assume that the hermitian form $\phi$ is split on $V$, i.e. there is a basis $e_1, \dots, e_n$ such that 
\[
\phi(ae_i,be_{n+1-j}) = \overline{a}b\delta_{i,j} \quad \text{for  all} \quad a,b \in K_1, 
\]
where $a \mapsto \overline{a}$ is the nontrivial element of $\text{Gal}(K_1/\QQ_p)$. Fix such a basis, and consider the self-dual lattice chain
\[
\calL_I=\{\Lambda_i\}_{i=\pm I+n\ZZ}
\]
in $V$, with $\Lambda_i$ as in (\ref{eq 212}) and where $I=\{\ell\}$ is a strongly non-special index (see Definition \ref{DefNonSpecial}).
Let $\mathcal{G}_I = \underline{{\rm Aut}}(\mathcal{L}_I)$ be the (smooth) group scheme over $\mathbb{Z}_p$ with $P_{\{\ell\}} = \mathcal{G}_I(\mathbb{Z}_p)$ and $G\otimes_{\mathbb{Z}_p}\mathbb{Q}_p $ as its generic fiber. Here, we denote by $P_{\{\ell\}}$ the stabilizer of $\Lambda_{\{\ell\}}$ in $G(\mathbb{Q}_p)$. As in \S \ref{ParahoricSbgrs}, when $n$ is odd the stabilizer $P_{\{\ell\}}$ is a parahoric subgroup. When $n$ is even, $P_{\{\ell\}}$ is not a parahoric subgroup since it contains a parahoric subgroup with index 2 and the corresponding parahoric group scheme is its connected component $ P_{\{\ell\}}^\circ$.  

Choose also a sufficiently small compact open subgroup $K^p$ of the prime-to-$p$ finite adelic points $G({\mathbb A}_{f}^p)$ of $G$ and set $\mathbf{K}=K^pP_{\{\ell\}}$ and $\mathbf{K}'=K^pP_{\{\ell\}}^\circ $. As pointed out in \cite[\S 1.3]{PR}, the Shimura varieties ${\rm Sh}_{\mathbf{K}'}(G, X)$ and ${\rm Sh}_{\mathbf{K}}(G, X)$ have isomorphic geometric connected components. Thus, from the point of view of constructing reasonable integral models, we may restrict our attention to ${\rm Sh}_{\mathbf{K}}(G, X)$; since $P_{\{\ell\}}$ corresponds to a lattice set stabilizer, this Shimura variety is given by a simpler moduli problem. The Shimura variety  ${\rm Sh}_{\mathbf{K}}(G, X)$ with complex points
 \[
 {\rm Sh}_{\mathbf{K}}(G, X)(\mathbb{C})=G(\mathbb{Q})\backslash X\times G({\mathbb A}_{f})/\mathbf{K}
 \]
is of PEL type and has a canonical model over the reflex field $K$. 

We consider the moduli functor $\mathcal{A}^{\rm naive}_{\mathbf{K}}$ over $\Spec \mathcal{O} $ given in \cite[Definition 6.9]{RZbook}:\\
A point of $\mathcal{A}^{\rm naive}_{\mathbf{K}}$ with values in the $\Spec \mathcal{O} $-scheme $S$ is the isomorphism class of quadruples $(A,\iota, \bar{\lambda}, \bar{\eta})$ consisting of:
\begin{enumerate}
    \item An object $(A,\iota)$, where $A$ is an abelian scheme with relative dimension $n$ over $S$ (terminology
of \cite{RZbook}), compatibly endowed with an
action of $\calO$: 
\[ \iota: \calO \rightarrow \text{End} \,A \otimes \mathbb{Z}_p.\] 
    
\item A $\mathbb{Q}$-homogeneous principal polarization $\bar{\lambda}$ on $(A,\iota)$ containing a polarization $\lambda$ such that $\ker \lambda\subset A[\iota(\pi)]$ of height $2\ell$.
    \item A $K^p$-level structure
    \[
\bar{\eta} : H_1 (A, {\mathbb A}_{f}^p) \simeq W \otimes  {\mathbb A}_{f}^p \, \text{ mod} \, K^p
    \]
which respects the bilinear forms on both sides up to a constant in $({\mathbb A}_{f}^p)^{\times}$ (see loc. cit. for
details).

The abelian scheme $A$ should satisfy the determinant condition (i) of loc. cit.
\end{enumerate}
(We refer the reader to loc.cit., 6.3–6.8 and \cite[\S 3]{P} for more details on the definitions of the terms employed here.) The functor $\mathcal{A}^{\rm naive}_{\mathbf{K}}$ is representable by a quasi-projective scheme over $\mathcal{O}$; recall that we assume that $K^p$ is sufficiently small. As in loc. cit., since the Hasse principle is satisfied for the unitary group we can see that there is a natural isomorphism
\[
\mathcal{A}^{\rm naive}_{\mathbf{K}} \otimes_{\calO} K_1 = {\rm Sh}_{\mathbf{K}}(G, X)\otimes_{K} K_1.
\]
Extending the abelian scheme $A$ and its dual abelian scheme $A^\vee$ periodically, we get an $\calL_I$-set of abelian schemes. Denote by $H^1_{dR}(A)$ the first de Rham cohomology sheaf, and let $M(A)=H^1_{dR}(A)^\vee$ be the dual of de Rham cohomology. Here $M(A)$ is a finite locally free $\calO_S$-module of rank $2n$. Thus we have the covariant Hodge filtration
\begin{equation}
0\rightarrow \omega_{A^\vee}\rightarrow M(A)\rightarrow \text{Lie}\,(A)\rightarrow 0,	
\end{equation}
where $\omega_{A^\vee}$ is a locally free $\calO_S$-module of rank $n$ (see \cite[\S 6.3, \S 14]{H} and \cite[\S 5.3]{Sm3} for more details). 
As is explained in \cite{RZbook} (see also \cite{P}), the naive local model ${\rm M}_I^{\rm naive}$ is connected to the moduli scheme $\mathcal{A}^{\rm naive}_{\mathbf{K}}$ via the local model diagram 
\[
\mathcal{A}^{\rm naive}_{\mathbf{K}} \ \xleftarrow{\psi_1} \Tilde{\mathcal{A}}^{\rm naive}_{\mathbf{K}} \xrightarrow{\psi_2} {\rm M}_I^{\rm naive}
\]
where the morphism $\psi_1$ is a $\mathcal{G}_I$-torsor and $\psi_2$ is a smooth and $\mathcal{G}_I$-equivariant morphism. Equivalently, there should be a relatively
representable morphism of algebraic stacks 
 \[
\phi:  \mathcal{A}^{\rm naive}_{\mathbf{K}} \to [  {\rm M}_I^{\rm naive}/ \mathcal{G}_I\otimes_{\mathbb{Z}_p}\mathcal{O}]
 \]
which is smooth of relative dimension $\text{dim}(G)$. Let us form the cartesian product of $\phi$ with the morphisms ${\rm M}^{\rm spl}_I \rightarrow \Mloc_I \hookrightarrow {\rm M}_I^{\rm naive}$, where ${\rm M}^{\rm spl}_I  $ and $\Mloc_I  $ are the splitting and local model respectively (see \S \ref{LocalModelVariantsSplit} for more details),
\[
\begin{matrix}
\mathcal{A}^{\rm spl}_{\mathbf{K}} &\longrightarrow & [{\rm M}^{\rm spl}_I/ \mathcal{G}_{I,\calO}] \\
\Big\downarrow && \Big\downarrow \\
\mathcal{A}^{\rm loc}_{\mathbf{K}} &\longrightarrow& [{\rm M}^{\rm loc}_I/ \mathcal{G}_{I,\calO}] \\
\Big\downarrow && \Big\downarrow \\
\mathcal{A}^{\rm naive}_{\mathbf{K}} &\longrightarrow & [{\rm M}^{\rm naive}_I/ \mathcal{G}_{I,\calO}] .
\end{matrix}
\]
The scheme $\mathcal{A}^{\rm loc}_{\mathbf{K}}$ is a closed subscheme of $\mathcal{A}^{\rm naive}_{\mathbf{K}}$ and is the image of $\mathcal{A}^{\rm spl}_{\mathbf{K}}$ in $\mathcal{A}^{\rm naive}_{\mathbf{K}}$. $\mathcal{A}^{\rm loc}_{\mathbf{K}}$ is a linear modification of $\mathcal{A}^{\rm naive}_{\mathbf{K}}$ in the sense of in the sense of \cite[\S 2]{P} (see also \cite[\S 15]{PR2}); similarly, $\mathcal{A}^{\rm spl}_{\mathbf{K}}$ is a linear modification of $\mathcal{A}^{\rm naive}_{\mathbf{K}}  $.
 \begin{Theorem}\label{SplLM}
     
 For every $K^p$ as above, there is a
 scheme $\mathcal{A}^{\rm spl}_{\mathbf{K}}$, flat over $\Spec(\calO)$, 
 with
 \[
\mathcal{A}^{\rm spl}_{\mathbf{K}}\otimes_{\calO} K_1={\rm Sh}_{\mathbf{K}}(G, X)\otimes_{K} K_1,
 \]
 and which supports a local model diagram
  \begin{equation}\label{LMdiagramReg1}
\begin{tikzcd}
&\wti{\mathcal{A}}^{\rm spl}_{\mathbf{K}}(G, X)\arrow[dl, "\pi^{\rm spl}_K"']\arrow[dr, "q^{\rm spl}_K"]  & \\
\mathcal{A}^{\rm spl}_{\mathbf{K}}  &&  {\rm M}_I^{\rm spl}
\end{tikzcd}
\end{equation}
such that:
\begin{itemize}
\item[a)] $\pi^{\rm spl}_{\mathbf{K}}$ is a $\mathcal{G}_I$-torsor 
and $q^{\rm spl}_{\mathbf{K}}$ is smooth and $\calG_I$-equivariant.

\item[b)] $\mathcal{A}^{\rm spl}_{\mathbf{K}}$ is normal and Cohen-Macaulay and has a reduced special fiber. 
\end{itemize}
 \end{Theorem}
\begin{proof}  
From the above we have that $\mathcal{A}^{\rm spl}_{\mathbf{K}}$ is a linear modification and this gives (a).  Using that ${\rm M}^{\rm spl}_I\rightarrow {\rm M}^{\rm loc}_I $ is projective from Corollary \ref{ProjMor}, we deduce that $\mathcal{A}^{\rm spl}_{\mathbf{K}} $ is represented by a scheme (see \cite[\S 2]{P}). Lastly, from Theorem \ref{SplitFlat} we deduce part (b).
\end{proof} 
Moreover, we can obtain an explicit moduli description of $\mathcal{A}^{\rm spl}_{\mathbf{K}}$ by adding in the moduli problem of $\mathcal{A}^{\rm naive}_{\mathbf{K}}$ an additional subspace in the Hodge filtration $ \omega_{A^\vee} \subset M(A)$.
More precisely, $\mathcal{A}^{\rm spl}_{\mathbf{K}}$ associates to an $\calO$-scheme $S$ the set of isomorphism classes of objects $(A,\iota, \bar{\lambda}, \bar{\eta},\mathscr{G}) $. Here $(A,\iota,\bar{\lambda}, \bar{\eta})$ is an object of  $\mathcal{A}^{\rm naive}_{\mathbf{K}}(S).$ 
The final ingredient $\mathscr{G}$ of an object of $\mathcal{A}^{\rm spl}_{\mathbf{K}}$ is the subspace $ \mathscr{G} \subset \omega_{A^\vee} \subset M(A) $ of rank one which satisfies the following conditions: 
\begin{equation}\label{eq 803}
\calO ~\text{acts by $\sigma_1$ on $\scrG$, and by $\sigma_2$ on $\omega_{A^\vee}/\scrG$},
\end{equation}
where $\Gal(K_1/\QQ_p)=\{\sigma_1=id, \sigma_2\}$. Note that condition (\ref{eq 803}) translates to condition (\ref{split.cond.}) by the isomorphism $M(A)\simeq \Lambda_{-\ell}\otimes \calO_S\simeq \Lambda_{n-\ell}\otimes \calO_S$. There is a forgetful morphism
\[
\tau :   \mathcal{A}^{\rm spl}_{\mathbf{K}} \longrightarrow \mathcal{A}^{\rm naive}_{\mathbf{K}}
\]
defined by $(A,\iota, \bar{\lambda}, \bar{\eta},\mathscr{G}) \mapsto (A,\iota,\bar{\lambda}, \bar{\eta}) $. 

 \begin{Theorem}\label{RegLM}
     
 For every $K^p$ as above, there is a
 scheme $\mathcal{A}^{\rm bl}_{\mathbf{K}}$, flat over $\Spec(\calO)$, 
 with
 \[
\mathcal{A}^{\rm bl}_{\mathbf{K}}\otimes_{\calO} K_1={\rm Sh}_{\mathbf{K}}(G, X)\otimes_{K} K_1,
 \]
 and which supports a local model diagram
  \begin{equation}\label{LMdiagramReg}
\begin{tikzcd}
&\wti{\mathcal{A}}^{\rm bl}_{\mathbf{K}}(G, X)\arrow[dl, "\pi^{\rm bl}_K"']\arrow[dr, "q^{\rm bl}_K"]  & \\
\mathcal{A}^{\rm bl}_{\mathbf{K}}  &&  {\rm M}_I^{\rm bl}
\end{tikzcd}
\end{equation}
such that:
\begin{itemize}
\item[a)] $\pi^{\rm spl}_{\mathbf{K}}$ is a $\mathcal{G}_I$-torsor 
and $q^{\rm spl}_{\mathbf{K}}$ is smooth and $\calG_I$-equivariant.

\item[b)] $\mathcal{A}^{\rm bl}_{\mathbf{K}}$ is regular and has special fiber which is a divisor with normal crossings. 
\end{itemize}
 \end{Theorem}
 \begin{proof}  
From the local model diagram (\ref{LMdiagramReg1}) we have that $q^{\rm spl}_{\mathbf{K}}:\wti{\mathcal{A}}^{\rm spl}_{\mathbf{K}}(G, X) \rightarrow {\rm M}_I^{\rm spl} $ is smooth and $\calG_I$-equivariant. We set
\[
\wti{\mathcal{A}}^{\rm bl}_{\mathbf{K}}=\wti{\mathcal{A}}^{\rm spl}_{\mathbf{K}}\times_{{\rm M}_I^{\rm spl}}{\rm M}_I^{\rm bl}
\]
which carries a diagonal $\mathcal{G}_I$-action. Since ${\rm M}_I^{\rm bl} \longrightarrow {\rm M}_I^{\rm spl}$ is given by a blow-up, is projective, and we can see  (\cite[\S 2]{P}) that the quotient
\[
\pi^{\rm reg}_K: \wti{\mathcal{A}}^{\rm bl}_{\mathbf{K}} \longrightarrow \mathcal{A}^{\rm bl}_{\mathbf{K}}:=\calG_I\backslash \wti{\mathcal{A} }^{\rm bl}_K(G, X)
\]
is represented by a scheme and gives a $\calG_I$-torsor. In fact, since blowing-up commutes with \'etale localization, $ \mathcal{A}^{\rm bl}_{\mathbf{K}}$ is the blow-up of $\mathcal{A}^{\rm spl}_{\mathbf{K}} $ along the locus of its special fiber where $\pi \omega_{A^\vee}=0$. The projection gives a smooth $\calG_I$-morphism
\[
q^{\rm reg}_{\mathbf{K}}: \wti{\mathcal{A}}^{\rm bl}_{\mathbf{K}} \longrightarrow {\rm M}_I^{\rm bl}
\]
which completes the local model diagram. Property (b) follows from Theorem \ref{ResolutionSpl} and property (a) which implies that $ \mathcal{A}^{\rm bl}_{\mathbf{K}}$ and
$ {\rm M}_I^{\rm bl}$ are locally isomorphic for the \'etale topology. 
\end{proof} 
\begin{Corollary}
$\mathcal{A}^{\rm bl}_{\mathbf{K}}$ is the blow-up of $\mathcal{A}^{\rm spl}_{\mathbf{K}} $ along the locus of its special fiber where the Hodge filtration $ w_{A^\vee}  $ is annihilated by the action of the uniformizer $\pi$. 
\end{Corollary}
\begin{proof}
    It follows from the proof of the above theorem. 
\end{proof}
\begin{Remarks}
{\rm 
\begin{enumerate}
    \item From the above discussion, we can obtain a semi-stable integral model for the Shimura variety ${\rm Sh}_{\mathbf{K}'}(G, X)$ where $\mathbf{K}'=K^pP_I^\circ $. In this case, the corresponding local models $\Mloc_I$ of ${\rm Sh}_{\mathbf{K}'}(G, X)$ agree with the Pappas-Zhu local models $ \mathbb{M}_{P_I^\circ}(G,\mu_{r,s})$ for the local model triples $(G,\{\mu_{r,s}\},P_I^\circ)$. (See \cite[Theorem 1.2]{PZ} and \cite[\S 8]{PZ} for more details.)
    \item Similar results can be obtained for corresponding Rapoport-Zink formal schemes. (See \cite[\S 4]{HPR} for an example of this parallel treatment.)
\end{enumerate}
}
\end{Remarks}

\section{Acknowledgments} We thank G. Pappas for several useful comments on a preliminary version of this article. The first author would like to thank W. Zhang for presenting this problem to him. We also thank the referee for their useful suggestions. The first author was supported by CRC 1442 Geometry: Deformations and Rigidity and the Excellence Cluster of the Universit\"{a}t M\"{u}nster.

\Addresses
\end{document}